\documentclass[reqno]{amsart}
\usepackage{graphicx}
\usepackage{amscd}
\usepackage{amsmath}
\usepackage{amsfonts}
\usepackage{amssymb}
\usepackage{cite}
\usepackage{xcolor}
\usepackage{booktabs}
\usepackage{geometry}
\def\R {\mathbb{R}}
\def\C {\mathcal{C}}

\def\c{\mathfrak{C}}

\def\d {\mathbf{D}}
\newtheorem{proposition}{Proposition}[section]
\newtheorem{theorem}[proposition]{Theorem}
\newtheorem{corollary}{Corollary}[section]
\newtheorem{lemma}{Lemma}[section]
\theoremstyle{definition}
\newtheorem{definition}{Definition}[section]
\newtheorem{remark}{Remark}[section]
\numberwithin{equation}{section}
\newtheorem{example}{Example}[section]

\begin{document}

\title[formula linking a function to the order]
{On the explicit formula linking a function\\ to the order of its
fractional derivative}

\author[ V.Semenov and  N. Vasylyeva]
{Vasyl Semenov and Nataliya Vasylyeva}

\address{Kyiv Academic University
\newline\indent
Vernadsky blvd.\ 36, 03142, Kyiv, Ukraine} \email[V.
Semenov]{vasyl\underline{\ }delta@gmail.com}

\address{Institute of Applied Mathematics and Mechanics of NAS of Ukraine
\newline\indent
G.Batyuka st.\ 19, 84100 Sloviansk, Ukraine;  and
\newline\indent
S.P. Timoshenko Institute of Mechanics of NAS of Ukraine
\newline\indent
Nesterov str.\ 3, 03057 Kyiv, Ukraine}
\email[N.Vasylyeva]{nataliy\underline{\ }v@yahoo.com}


\begin{abstract}
In this paper, given a certain regularity of a function $v$, we
derive an explicit formula relating the order $\nu_0\in(0,1)$ of the
leading fractional derivative in a fractional differential operator
$\mathbf{D_t}$ with the variable coefficients $r_i=r_i(x,t)$ and the
function $v$ on which this operator acts. Moreover, we discuss
application of this result in the  reconstruction of the memory
order of  semilinear subdiffusion with memory terms. To achieve this
aim, we analyze some inverse problems to multi-term  fractional in
time  ordinary and partial differential equations with smooth  local
or nonlocal additional measurements for small time. In conclusion,
we discuss how this formula may be exploited to numerical
computation of $\nu_0$ in the case of discrete noisy observation in
the corresponding inverse problems. Our theoretical results along
with the computational algorithm are supplemented by numerical
tests.
\end{abstract}

\maketitle

\section{Introduction}
\label{si}

\noindent The fractional derivatives are very effective tool in the
description of memory or delay phenomena, which are structurally
present in real-life models arising in continuum mechanics,
thermodynamics, medicine, biology and so on (see for example,
\cite{FLN,HPSV,HZL,KE,PKLS}). This leads in a natural way to the
study of fractional differential equations (FDEs).

In this paper, we focus on the nonlocal operator having the form
\begin{equation}\label{i.1}
\d_t=\begin{cases} \sum\limits_{i=0}^{M}r_i(t)\d_{t}^{\nu_i},\quad
\text{the I type FDO},\\
 \sum\limits_{i=0}^{M}\d_{t}^{\nu_i}r_i(t),\quad
\text{the II type FDO},
\end{cases}
\quad 0<\nu_M<...<\nu_1<\nu_0<1,
\end{equation}
with given variable coefficients $r_{i}=r_{i}(t),$ $i=0,..,M,$ which
will be specified in Section \ref{s3}. The symbol $\d_{t}^{\nu_i}$
stands for the regularized (left) fractional Caputo derivative of
order $\nu_i$ in time, defined as
\begin{equation}\label{i.2}
\d_{t}^{\nu_i}v(t)=\begin{cases}
\frac{1}{\Gamma(1-\nu_i)}\frac{d}{dt}\int\limits_{0}^{t}\frac{v(s)-v(0)}{(t-s)^{\nu_i}}ds,\qquad
\nu_i\in(0,1),\\
\frac{d}{dt}v(t),\qquad\qquad\qquad\qquad\quad \nu_i=1,
\end{cases}
\end{equation}
where $\Gamma$ is the Euler Gamma-function.
 In the case of time and space depending $v$, the ordinary time
derivative  in this definition is replaced by the corresponding
partial derivative.

 Clearly,  $\nu_0$ in  \eqref{i.1}
is the order of the leading derivative in the operator $\d_t$. By
analogy with differential operators of integer order, $\d_t$ will be
called a fractional differential operator (FDO) in time  of order
$\nu_0$ with variable coefficients (if at least one of $r_i$ is
time-dependent).  We notice that, since the classical Leibniz rule
does not work in the case of the fractional derivative, in
particular, of the Caputo derivative, the II type FDO is more
complex than $\d_t$ of the first type. Indeed, following
\cite[Proposition 5.5]{SV}, we have the relation
\[
\d_t^{\nu_i}(r_i(t)v(t))=r_{i}(t)\d_{t}^{\nu_i}v(t)+v(0)\d_{t}^{\nu_i}r_i(t)+\frac{\nu_i}{\Gamma(1-\nu_i)}\int_{0}^{t}\frac{r_i(t)-r_i(s)}{(t-s)^{\nu_i+1}}
[v(s)-v(0)]ds
\]
telling that even if $r_i$ admits a continuous fractional derivative
$\d_{t}^{\nu_i},$ the last term in the right-hand side of this
equality is a convolution with a stronger singular kernel
$t^{-\nu_i-1}$ and, hence,  cannot be considered (generally
speaking) as a minor term in the further study.

An intriguing feature of a function $v(t)$ having a H\"{o}lder
continuous Caputo derivative with respect to time of order
$\nu\in(0,1)$ is the following explicit relations between the order
$\nu$ and the function $v(t)$ stated in \cite[Lemma 10.2]{KPSV2} and
\cite[Lemma 4.1]{KPSV1},
\begin{equation}\label{f1}
\nu=\underset{t\to 0}{\lim}\frac{\ln|v(t)-v(0)|}{\ln\,t}
\end{equation}
and
\begin{equation}\label{f2}
\nu=\underset{t \to
0}{\lim}\frac{t[v(t)-v(0)]}{\int_{0}^{t}[v(t)-v(0)]d\tau}-1,
\end{equation}
which hold whenever $\d_t^{\nu}v|_{t=0}\neq 0$. We recall that the
crucial point in deriving these formulas (see \cite{KPSV1,KPSV2}) is
the asymptotic representation
\begin{equation}\label{f0}
v(t)-v(0)=\frac{\d_t^{\nu}v|_{t=0}}{\Gamma(1+\nu)}t^{\nu}+o(t^{\nu})\qquad\text{if}\qquad
t<<1.
\end{equation}
The importance of these formulas comes from their application in the
inverse problems related with reconstruction of the fractional order
for the corresponding fractional differential equations studied in
\cite{J,HNWY,KPSV1,KPSV2,Po}.

In this connection, there appear the following very natural
questions.

\noindent$\bullet$ \textit{What happens if instead of behavior
$\d_t^{\nu}v|$ at $t=0$, we will have information about
$\d_tv|_{t=0}$?}

\noindent$\bullet$ \textit{Will formulas \eqref{f1}-\eqref{f2} still
work in this case and if yes, what assumptions on $r_i$ and $v$ are
needed?}

The answers to these questions were partially found in \cite{PSV},
where assuming
\begin{equation*}\label{f3}
\d_t v|_{t=0}\neq 0
\end{equation*}
and requiring H\"{o}lder continuity of  $\d_t^{\nu_i}v$ and $r_i$
(for the I type FDO) and of  $\frac{dr_i}{dt},$ $i=0,1,..,M,$ (for
the II type FDO), the authors established the relations similar to
\eqref{f1}
\begin{equation*}\label{f4}
\nu_0=
\begin{cases}
\underset{t\to 0}{\lim}\frac{\ln|v(t)-v(0)|}{\ln t}\qquad\qquad
\text{for
the I type FDO,}\\
\underset{t\to 0}{\lim}\frac{\ln|r_0(t)v(t)-r_0(0)v(0)|}{\ln
t}\qquad \text{for the II type FDO.}
\end{cases}
\end{equation*}
To obtain these equalities, the more complex asymptotic than
\eqref{f0} was needed. In particular, in the case of the I type FDO,
the authors used
\begin{equation}\label{f5}
[v(t)-v(0)][r_0(0)\Gamma(1+\nu_0)+\sum_{i=1}^{M}r_i(0)t^{\nu_0-\nu_i}\Gamma(1+\nu_i)]
=t^{\nu_0}\d_tv|_{t=0}+\sum_{i=0}^{M}r_i(0)t^{\nu_0-\nu_i}o(t^{\nu_i})\quad\text{as}\quad
t<<1.
\end{equation}

Thus, the questions concerning formula \eqref{f2} are still open,
and in this  paper, we fill this gap deriving the formulas
\begin{equation}\label{i.3}
\nu_0=
\begin{cases}
\underset{t\to
0}{\lim}\frac{t[v(t)-v(0)]}{\int\limits_{0}^{t}[v(\tau)-v(0)]d\tau}-1\quad\qquad\text{for
the I type FDO,}\\
\underset{t\to
0}{\lim}\frac{t[r_0(t)v(t)-r_0(0)v(0)]}{\int\limits_{0}^{t}[r_0(\tau)v(\tau)-r_0(0)v(0)]d\tau}-1\quad\text{for
the II type FDO}.
\end{cases}
\end{equation}
which are true whenever $\d_t v|_{t=0}$ does not vanish and $v,r_i$
possess the certain regularity.

In our analysis, in opposite to  \cite{PSV}, we do not exploit
asymptotic \eqref{f5}. Instead of this, under certain assumptions on
the function $v$ and coefficients $r_i$, we first prove the
following equivalence in the behavior of $\d_t v$ and
$\d_t^{\nu_0}v$ at $t=0$,
\[
\d_t^{\nu_0}v|_{t=0}\neq 0\quad\Longleftrightarrow\quad \d_t
v|_{t=0}\neq \begin{cases} 0\qquad\qquad\qquad\qquad\text{for the I
type
FDO,}\\
v(0)\sum\limits_{i=1}^{M}\d_t^{\nu_i}r_i(0)\qquad \text{for the II
type FDO}.
\end{cases}
\]
The latter, in turn, allows utilizing the simpler asymptotic
\eqref{f0} and, hence, leads to the desired results.

 In
connection with the straightforward application of formulas
\eqref{i.3} to fractional differential equations, we are able to
link a local classical solution of an
 ordinary differential equation with the fractional differential operator
\eqref{i.1}  and with a memory (convolution) term  with the order
$\nu_0$.

Turning to another achievement of this work,  we first recall that
the characteristics of anomalous diffusion include history
dependence (memory term), long-range (or nonlocal) correlation in
time and heavy-tail characteristics, while its distinguishing
feature is that the mean-square displacement of diffusion species
$\langle\Delta x\rangle^2$ scales as a nonlinear power law in time,
i.e., $\langle\Delta x\rangle^2\approx t^{\nu},$ $\nu>0,\nu\neq 1$.
If the anomalous diffusion exponent $\nu$ belongs to the interval
$(0,1),$ then the underlying diffusion process is called
subdiffusive and, accordingly, order $\nu$ is called (memory) order
of subdiffusion. The constitutive relations in the thermoelasticity,
viscoelastic materials and diffusion process in biological tissues
are successfully described by both  partial or ordinary equations
with a single fractional derivative and the equations with
fractional differential operators having form \eqref{i.1}.
 In this connection, another important application of \eqref{i.3}, as anticipated
above, deals with reconstruction of the memory order in semilinear
multi-term fractional in time diffusion  with memory terms via an
additional local or nonlocal measurement. More precisely, for given
bounded domain $\Omega\subset\R^{n}$ with a smooth boundary
$\partial\Omega $ (at least $\partial\Omega\in\C^{2+\alpha},$
$\alpha\in(0,1)$), and given arbitrarily finite time $T>0$, we set
\[
\Omega_T=\Omega\times(0,T)\qquad\text{and}\qquad
\partial\Omega_T=\partial\Omega\times[0,T].
\]
For fixed $\nu_i\in(0,1),$ we focus on the identification $\nu_0$ in
$\d_t$ in  the following equation with unknown
$u=u(x,t):\Omega_T\to\R,$
\begin{equation}\label{i.4}
\d_tu-\mathcal{L}_1u-\mathcal{K}*\mathcal{L}_2 u+g(u)=g_0(x,t),
\end{equation}
where the summable convolution (memory) kernel $\mathcal{K}$ and the
functions $g,g_0$ are given. As for the operators involved,
$\mathcal{L}_i$ are the uniform elliptic operators of the second
order with time and space dependent smooth coefficients, whose
precise form will be specified in Section \ref{s4}. As for $\d_t$,
this operator is defined via \eqref{i.1} with $r_i=r_i(x,t)$.

To reconstruct $\nu_0$, we analyze the corresponding inverse
problems subject  to the additional local or nonlocal observation
$\psi(t)$ for small time $t\in[0,t^*],$ $t^*<<\min\{1,T\}$. The
first  so called  local observation (LO) is observed at a spatial
point $x_0\in\bar{\Omega}$
\begin{equation}\label{i.5}
u(x_0,t)=\psi(t)\qquad\text{for all}\qquad t\in[0,t^*],
\end{equation}
while the second  called the nonlocal observation (NLO) is given as
\begin{equation}\label{i.6}
\int_{\Omega}u(x,t)dx=\psi(t) \qquad\text{for all}\qquad
t\in[0,t^*].
\end{equation}
It is worth noting that finding fractional order in \eqref{i.1},
\eqref{i.4} is  one of the most important inverse problems in the
papers related to mathematical models of anomalous phenomenon, and
has been extensively discussed, see for example recent survey
\cite{LLYa}. In the current literature, there are a lot of articles
devoted to the study of inverse problems on the recovery of order
$\nu_0$ via local or nonlocal observation in the case of one- or
multi-term $\d_t$ similar to \eqref{i.1} and concerning the
equations like \eqref{i.4} (see e.g.
\cite{HPV,HNWY,J,JK,JiK,ICM,KPSV1,KPSV2,Po,PSV1,SLJ,ZJY} and
references therein). In particular, the unique reconstruction of
$\nu_0$ via LO or NLO are discussed in
\cite{KPSV1,KPSV2,PSV1,J,JK,LY,WQQR}, where additional measurements
are made either for small time  (see \eqref{i.5} or \eqref{i.6}) or
on the whole time interval $[0,T]$. The issue concerning the
stability of the recovered $\nu_0$ is discussed in smooth fractional
H\"{o}lder classes and Sobolev or Sobolev-Slobodeckii spaces in
\cite{JiK,HPV,HPSV,LHY}. The influence of noisy observation on the
computation of $\nu_i$ is described in \cite{HNWY,KPSV1,KPSV2,PSV1}.
At last, we mention that the similar questions related to
simultaneous recovery of several scalar parameters in $\d_t$
including $\nu_i,$ $r_i$ are explored by various approaches and
techniques in \cite{JiK,LHY,HPV,HPSV}.

The above works follow two main different conceptions to find the
order of the leading fractional derivative. The first  started from
the pioneering work \cite{HNWY} deals with obtaining explicit
formulas for $\nu_0$ in the term of a measurement (for small or
large time) \cite{J,KPSV1,KPSV2,HPV,HPSV,PSV1,Po}. The second way
originated from \cite{ICM} concerns with the minimization of a
certain functional depending on both the solution of the
corresponding direct problems and given measurement either for the
terminal time $t=T$ or on the whole time interval $[0,T]$
\cite{JK,JiK,ICM,SLJ,ZJY,WQQR}.

The main evident advantage of the first approach is that the
computations by explicit formulas require only information about the
observations, while the second method needs not only the
measurements but also knowledge of the coefficients in the operators
in \eqref{i.4} and the corresponding right-hand sides.

In this work, inheriting the first conception, we deduce that the
memory order $\nu_0$ in \eqref{i.4} is computed via formulas
\eqref{i.3}, where $v=\psi(t)$ only if certain assumptions for the
model hold. Moreover, we prove that the corresponding inverse
problems have unique solution.

Our numerical experiments demonstrate that formulas \eqref{i.3} give
more accurate computational outcomes than another formulas. In
conclusion, we discuss numerical algorithm of reconstructing $\nu_0$
which is based on \eqref{i.3}, but in the case of discrete noisy
data.

\subsection*{Outline of the paper} The paper is organized as
follows: in Section \ref{s2}, we introduce the notation and the
functional setting. Our first achievements, presented  in the
Theorems \ref{t3.1}-\ref{t3.2} and the corresponding related
consequences, concern the linking formulas \eqref{i.3} and  are
stated in Section \ref{s3}, whereas in Section \ref{s4} we state the
results dealing with inverse problems for \eqref{i.4}-\eqref{i.6}.
The proofs of these results are carried out in Sections
\ref{s5}-\ref{s6}. Section \ref{s7} is devoted to numerical
simulations, where the numerical algorithm of reconstruction based
on the Tikhonov regularization scheme and quasi-optimality approach
are described in Section \ref{s7.1}, while the corresponding
numerical tests demonstrating performance of this algorithm  are
given in Section \ref{s7.2}.

\section{Functional Spaces and Notation}\label{s2}
\noindent

Throughout this paper, the symbol $C$ will denote a generic positive
constant, depending only on the structural values of the model. In
our study, for any nonnegative integer $l$ and for any $p\geq 1,$
$\alpha\in(0,1),$ and any Banach space $(X,\|\cdot\|_{X}),$ we will
exploit the usual spaces
\[
\C([0,T],X),\quad \C^{l+\alpha}(\bar{\Omega}),\quad
\C^{l+\alpha}([0,T]),\quad L^{p}(0,T),
\]
and for $l=0,1,2,$ and $\nu\in(0,1),$ we will use the fractional
H\"{o}lder spaces
\[
\C_{\nu}^{\alpha}([0,T])\qquad\text{and}\qquad
\C^{l+\alpha,\frac{l+\alpha}{2}\nu}(\bar{\Omega}_{T}).
\]
We refer a reader to \cite[Section2]{KPV} for a more detailed
discussion on these classes, but here below we  only recall the
definition of these fractional spaces.
\begin{definition}\label{d2.1}
The functions $v=v(t)$ and $u=u(x,t)$ belong to the classes
$\C_{\nu}^{\alpha}([0,T])$ and
$\C^{l+\alpha,\frac{l+\alpha}{2}\nu}(\bar{\Omega}_{T})$ for
$l=0,1,2,$ respectively, if $v$ and $u$
 together with their corresponding derivatives are continuous on
 $\bar{\Omega}_{T}$ and $[0,T],$ respectively, and the norms below
 are finite:
\begin{align*}
\|v\|_{\C_{\nu}^{\alpha}([0,T])}&=\|v\|_{\C([0,T])}+\|\mathbf{D}_{t}^{\nu}v\|_{\C([0,T])}+\langle\mathbf{D}_{t}^{\nu}v\rangle_{t,[0,T]}^{(\alpha)},
\\
\|u\|_{\C^{l+\alpha,\frac{l+\alpha}{2}\nu}(\bar{\Omega}_{T})}&=
\begin{cases}
\|u\|_{\C([0,T],\C^{l+\alpha}(\bar{\Omega}))}+\sum\limits_{|j|=0}^{l}\langle
D_{x}^{j}u\rangle_{t,\Omega_{T}}^{(\frac{l+\alpha-|j|}{2}\nu)},\qquad\qquad\qquad\qquad\qquad
l=0,1,\\
\|u\|_{\C([0,T],\C^{2+\alpha}(\bar{\Omega}))}+\|\d_{t}^{\nu}u\|_{\C^{\alpha,\frac{\nu\alpha}{2}}(\bar{\Omega}_{T})}+\sum\limits_{|j|=1}^{2}\langle
D_{x}^{j}u\rangle_{t,\Omega_{T}}^{(\frac{2+\alpha-|j|}{2}\nu)},\qquad
l=2,
\end{cases}
\end{align*}
where $\langle \cdot\rangle_{t,\Omega_{T}}^{(\alpha)}$, $\langle
\cdot\rangle_{t,[0,T]}^{(\alpha)}$ and $\langle
\cdot\rangle_{x,\Omega_{T}}^{(\alpha)}$ stand for the standard
H\"{o}lder seminorms of a function  with respect to time and space
variables, respectively.
\end{definition}
\noindent In a similar way, for $l=0,1,2,$ the spaces
$\C^{l+\alpha,\frac{l+\alpha}{2}\nu}(\partial\Omega_{T})$ are
defined.

Moreover, in this paper, we need the Hilbert space
$L^2_{\varrho}(t_1,t_2)$ of real-valued square integrable functions
with a positive weight $\varrho=\varrho(t)$ on $[t_1,t_2]$. The norm
and the inner product in this space are defined as
\[
\|v\|^{2}_{L^{2}_{\varrho}(t_1,t_2)}=\int_{t_1}^{t_2}\varrho(t)v^2(t)dt,\qquad
\langle
u,v\rangle_{L^{2}_{\varrho}(t_1,t_2)}=\int_{t_1}^{t_2}\varrho(t)u(t)v(t)dt.
\]

We conclude this section recalling the definition of the fractional
integrals and its very useful property.  Throughout the paper, for
any positive $\theta,$ we denote
\[
\omega_{\theta}(t)=\frac{t^{\theta-1}}{\Gamma(\theta)}.
\]
Then, we define the fractional Riemann-Liouville integral of order
$\theta$ of a function $v(t)$ with respect to time as
\[
I_{t}^{\theta}v(t)=(\omega_{\theta}*v)(t),
\]
where the symbol $*$ stands for the usual time-convolution product
\[
(\eta_1*\eta_2)(t)=\int_{0}^{t}\eta_1(t-\tau)\eta_2(\tau)d\tau.
\]
In conclusion, we recall the very useful inequality established in
\cite[Proposition 5.1]{KPSV}
\begin{equation}\label{2.1}
\|I_{t}^{\theta}[\d_{t}^{\theta}v-\d_t^{\theta}v(0)]\|_{\C^{\theta+\alpha}([0,T])}\leq
C\|v\|_{\C^{\alpha}_{\theta}([0,T])},
\end{equation}
if $v\in\C_{\theta}^{\alpha}([0,T])$ with $\theta,\alpha\in(0,1),
\theta+\alpha< 1$.


\section{Explicit Formula for the Order $\nu_0$}\label{s3}

\noindent We start by stipulating the set of requirements on the
structural quantities appearing in $\d_tv$ (see \eqref{i.1}).
\begin{description}
\item[h1. Assumptions on the given function] We require that for
$\nu_i$ satisfying \eqref{i.1}, $v,\d_t^{\nu_i}v\in\C([0,T]),$
$i=0,1,...,M$. Besides, there exists positive time $T^*\leq T$ such
that
\[
v\in\C_{\nu_0}^{\nu^*}([0,T^*])\qquad\text{and}\qquad
\d_{t}^{\nu_i}v\in\C^{\nu^*}([0,T^*])
\]
for some $\nu^*\in(0,1)$.
\item[h2. Assumptions on the coefficients in the I type FDO] For all
$t\in[0,T],$
\begin{equation}\label{3.0}
r_0(t)\neq 0,
\end{equation}
and $r_{i}\in\C([0,T]),$ $i=0,...,M$.
\item[h3. Assumptions on the coefficients in the II type FDO] We
assume that all $r_i$ are continuous on $[0,T]$, and there exists a
positive $T^*\leq T$, such that each
  $r_i$ has  the one of the following
regularities

\noindent(i) either $r_i\in\C^1([0,T^*])$;

\noindent(ii) or if $\nu*+\nu_i<1$, then
$r_i\in\C^{\mu_i}([0,T^*])\cap\C_{\nu_i}^{\mu^*}([0,T^*])$ for some
$\mu_i\in(\nu^*+\nu_i,1),$ $\mu^*\in(0,1);$

\noindent(iii) or if $1\leq\nu*+\nu_i<2$, then
$r_i\in\C^{\mu^*}([0,T^*])$ and
$\d_{t}^{\nu_i}r_i\in\C^{\mu_1^*}([0,T^*])$ with some
$\mu_i\in(\nu_i,1)$ and $\mu^*,\mu_1^*\in(0,1)$.
\end{description}

Our first result relates the order $\nu_0$ of the I type FDO ($\d_t
v$) with the function $v$.
\begin{theorem}\label{t3.1}
Let $\d_t$ be the I type FDO and let assumptions h1 and h2 hold. If
$\d_tv|_{t=0}\neq 0$, then
\begin{equation}\label{3.1}
\nu_0=\underset{t\to
0}{\lim}\frac{t[v(t)-v(0)]}{\int_{0}^{t}[v(\tau)-v(0)]d\tau}-1,
\end{equation}
and, besides, for all $t\in[0,T^*]$ the representation
\begin{equation}\label{3.1*}
v(t)=v(0)+\c_0t^{\nu_0}+\frac{1}{\Gamma(\nu_0)}\int_{0}^{t}(t-\tau)^{\nu_0-1}[\d_{\tau}^{\nu_0}v(\tau)-\d_{\tau}^{\nu_0}v(0)]d\tau
\end{equation}
 holds with
\[
\c_0=\frac{\d_tv(0)}{r_0(0)\Gamma(1+\nu_0)}.
\]
\end{theorem}
The analogous  result takes place in the case of $\d_t$ being the II
type FDO.
\begin{theorem}\label{t3.2}
Let $\d_t$ be the II type FDO and let h1, h3 and \eqref{3.0} hold.
If, additionally,
$$\d_tv|_{t=0}-v(0)\sum_{i=1}^{M}\d_t^{\nu_i}r_i(0)\neq 0,$$ then
\begin{equation}\label{3.2}
\nu_0=\underset{t\to
0}{\lim}\frac{t[r_0(t)v(t)-r_0(0)v(0)]}{\int_{0}^{t}[r_0(\tau)v(\tau)-r_0(0)v(0)]d\tau}-1,
\end{equation}
and, besides, for all $t\in[0,T^*],$ there is the equality
\begin{equation}\label{3.2*}
r_0(t)v(t)=r_0(0)v(0)+\c_1t^{\nu_0}+\frac{1}{\Gamma(\nu_0)}\int_{0}^{t}(t-\tau)^{\nu_0-1}[\d_{\tau}^{\nu_0}r_0(\tau)v(\tau)-\d_{\tau}^{\nu_0}r_0(0)v(0)]d\tau
\end{equation}
with
\[
\c_1=\frac{\d_tv|_{t=0}-v(0)\sum_{i=1}^{M}\d_t^{\nu_i}r_i(0)}{\Gamma(1+\nu_0)}.
\]
\end{theorem}
\begin{remark}\label{r3.0}
Our preliminary observation in Section \ref{s5} (see also Remark
\ref{r5.2} therein) tells us that if for each $i=0,1,...,M,$ $r_i$
meets requirement either (i) in h3 or (iii) in h3, then
$\d_t^{\nu_i}r_i(0)=0$. Thus, in this case, we just require
$\d_tv|_{t=0}\neq 0$  in Theorem \ref{t3.2}, and, accordingly,
$\c_1$ is rewritten as
$\c_{1}=\frac{\d_tv|_{t=0}}{\Gamma(1+\nu_0)}$.
\end{remark}
\begin{remark}\label{r3.a}
Clearly, representations \eqref{3.1*} and \eqref{3.2*} describe the
asymptotic behavior of the function $v$ at $0$. Namely, appealing to
the regularity of $v$ and $r_i$ stated in h1--h3 and performing the
straightforward calculations, we immediately conclude that
\begin{equation*}\label{a.1}
v(t)=v(0)+t^{\nu_0}\c_0+o(t^{\nu_0})\quad\text{as}\quad t\to 0
\end{equation*}
in the case of the I type FDO, while
\begin{equation*}\label{a.2}
v(t)=\frac{r_0(0)v(0)}{r_0(t)}+t^{\nu_0}\frac{\c_1}{r_0(t)}+o(t^{\nu_0})\quad\text{as}\quad
t\to 0
\end{equation*}
in the case of the II type FDO.
\end{remark}
\begin{remark}\label{r3.1}
It is worth mentioning that all statements in Theorems
\ref{t3.1}-\ref{t3.2}, with slight modifications, hold in the case
of $r_i$ and $v$ depending not only on time but also on the spatial
variables. Namely, in the case of  the I type FDO, we assume that
$r_i=r_i(x,t):\Omega_T\to\R$ and $v=v(x,t):\Omega_T\to\R$ such that
\[
v\in\C(\bar{\Omega}_{T})\cap\C(\bar{\Omega},\C_{\nu_0}^{\nu^*}[0,T^*]),\quad
\d_t^{\nu_i}v\in\C(\bar{\Omega}_{T})\cap\C(\bar{\Omega},\C^{\nu^*}[0,T^*]),\quad
r_i\in\C(\bar{\Omega}_{T}).
\]
The similar requirements hold in the case of the II type FDO, and
their particular case will be described in Section \ref{s4}.
\end{remark}
Next, we proceed with a straightforward application of Theorems
\ref{t3.1}--\ref{t3.2} to the Cauchy problem for the multi-term
fractional in time ordinary differential equation in unknown
$v=v(t)$,
\begin{equation}\label{3.3}
\begin{cases}
\d_tv+\mathcal{K}*v+v=f_0(t)+f(t,v),\quad t>0,\\
v(0)=v_0.
\end{cases}
\end{equation}
Here, $\mathcal{K}$ and $f_0,f$ represent a given convolution
(memory) summable kernel and the external sources, the latter
depending (possibly in a nonlinear way) on the variable $v$.

At this point, we state main assumption on the given functions in
\eqref{3.3}.
\begin{description}
\item[h4. Regularity of the given functions] For given $T>0$, we
assume that
\[
\mathcal{K}\in L^1(0,T),\qquad f_0\in\C([0,T])\qquad \text{and}\quad
f\in\C([0,T]\times \R).
\]
\end{description}
\begin{lemma}\label{l3.1}
Let $T>0$ be given and h2, h4 hold. In the case of the II type FDO,
we  additionally require h3. Assume that $v$ is a unique solution of
\eqref{3.3} for $t\leq T^*]$ and, besides, there exists some
positive $T^*\leq T$ such that $v$ satisfies \eqref{3.3} in the
classical sense for all $t\in[0,T^*]$ and
$v\in\C_{\nu_0}^{\alpha^*}([0,T^*]),$ for some $\alpha^*\in(0,1)$
and $\d_{t}^{\nu_i}v\in\C^{\alpha^*}([0,T^*]),$ $i=1,...,M$.

If
\begin{equation*}\label{3.4}
f_0(0)+f(0,v_0)-v_0\neq 0\quad \text{for the I type FDO,}
\end{equation*}
and
\begin{equation*}\label{3.5}
f_0(0)+f(0,v_0)-v_0-v_0\sum_{i=1}^{M}\d_t^{\nu_i}r_i(0)\neq 0\quad
\text{for the II type FDO,}
\end{equation*}
then $v_0$ is computed via \eqref{3.1} for the I type FDO and
\eqref{3.2} for the II type FDO.

\noindent Moreover, the solution $v$ admits representations
\eqref{3.1*} and \eqref{3.2*}, respectively.
\end{lemma}
\begin{remark}\label{r3.3}
We finally observe that conditions h2-h4 on the smoothness of the
given functions in \eqref{3.3} and on the behavior of the
nonlinearity do not provide, in general, the (local) classical
solvability of the Cauchy problem on $[0,T^*].$ Nevertheless, these
assumptions are sufficient to prove results as stated in Lemma
\ref{l3.1}. Moreover, in Section  \ref{s6} (see Remark \ref{r6.1}),
we will specify additional conditions on $f,f_0$, which guarantee
the desired solvability of \eqref{3.3}.
\end{remark}

We complete this section with important properties of the solution
$v$ to \eqref{3.3}, which is stated in Lemma \ref{l3.1} and Remark
\ref{r3.1}.
\begin{corollary}\label{c3.1}
Under assumptions of Lemma \ref{l3.1}, the unique solution $v$
\eqref{3.3} admits asymptotic behavior \eqref{3.1*} for the I type
FDO and \eqref{3.2*} for the II type FDO as $t\to 0$ with
\[
\c_0=\frac{f_0(0)+f(0,v_0)-v_0}{r_0(0)\Gamma(1+\nu_0)}\quad\text{and}\quad
\c_1=r_0(0)\c_0-\frac{v_0\sum_{i=1}^{M}\d_{t}^{\nu_i}r_i(0)}{\Gamma(1+\nu_0)}.
\]
\end{corollary}
In conclusion, we notice that, Lemma \ref{l3.1} is verified via a
straightforward application of Theorems \ref{t3.1} and \ref{t3.2} to
the solution of \eqref{3.3}, so we omit it here. As for the proof of
Theorems \ref{t3.1}--\ref{t3.2}, it is rather technical and will be
postponed to the forthcoming Section \ref{s5}.


\section{Reconstruction of the Memory Order in Multi-term Semilinear Subdiffusion}\label{s4}

\noindent In this section, we focus on the another application of
Theorems \ref{t3.1}-\ref{t3.2}, which deals with the reconstruction
of the order of FDO \eqref{i.1} in the semilinear subdiffusion
equation \eqref{i.4} via additional local or nonlocal measurement
(see \eqref{i.5}-\eqref{i.6}). In order to achieve this aim, we
analyze the following two inverse problems. The first concerns with
 the reconstruction of the order $\nu_0$ under the assumption
that $u$ is a unique solution of \eqref{i.4} satisfying either local
or nonlocal observation on $[0,t^*]$. The second inverse problem
deals with finding couple $(\nu_0,u)$, where  $u$ is a global
classical solution of   \eqref{i.1}, \eqref{i.4} subject to the
corresponding initial and boundary conditions and, besides, $u$
satisfies either \eqref{i.5} or \eqref{i.6}. To complete the
statement of the inverse problems related to the identification of
$\nu_0$, we assume that equation \eqref{i.4} is supplemented with
the initial condition
\begin{equation}\label{4.1}
u(x,t)=u_0(x_0)\qquad \text{in}\qquad\bar{\Omega},
\end{equation}
while in the case of finding $(\nu_0,u)$, we need additionally  one
of the following boundary conditions on $\partial\Omega_T$:
\smallskip
\begin{itemize}
\item[(i)] Dirichlet boundary condition (\textbf{DBC})
\begin{equation}\label{4.2}
u(x,t)=0,
\end{equation}
\item[(ii)] Boundary condition of the third kind (\textbf{3BC})
\begin{equation}\label{4.3}
\mathcal{N}_1 u+\mathcal{K}_{1}*\mathcal{N}_{2}u=\varphi_{1},
\end{equation}
\item[(iii)] Fractional dynamic  boundary condition  (\textbf{FDBC})
\begin{equation}\label{4.4}
\mathbf{D}_{t}u -\mathcal{N}_1
u+\mathcal{K}_{1}*\mathcal{N}_{2}u=\varphi_{2}.
\end{equation}
\end{itemize}
Here, the functions $\varphi_i=\varphi_i(x,t),$ $i=1,2,$ and
$u_0=u_0(x)$ are specified  in h10 below, as well as a summable
memory kernel $\mathcal{K}_1$ (see h8). As for the operators
$\mathcal{N}_i$, they are first-order differential operators, whose
precise form will be added below.

At this point, for reader's convenience, we recall statement of the
inverse problem concerning with the recovery of the couple
$(\nu_0,u)$ via the local or nonlocal measurement.

\noindent\textit{Statement of the inverse problem:} for given
right-hand sides in \eqref{i.4}-\eqref{i.6} and
\eqref{4.1}-\eqref{4.4}, coefficients of the operators $\d_t,$
$\mathcal{L}_i,$ $\mathcal{N}_i$, the orders $\nu_i,$ $i=1,...,M,$
and memory kernels $\mathcal{K},$ $\mathcal{K}_1,$ the inverse
problem consists of finding the pair $(\nu_0,u)$ through local
observation \eqref{i.5} or nonlocal one \eqref{i.6} for small time
$[0,t^*]$, where $\nu_0$ is the order of the fractional differential
operator $\d_t$ and $u$ is a unique (global) classical solution to
the direct problem \eqref{i.4}, \eqref{4.1}-\eqref{4.6} satisfying
\eqref{i.5} or \eqref{i.6}, respectively, for $t\in[0,t^*].$

Next, we stipulate the set of hypotheses on the structural
quantities appearing in \eqref{i.1}, \eqref{i.4}-\eqref{i.6} and
\eqref{4.1}-\eqref{4.4}.
\begin{description}
\item[h5. Conditions on the operators $\mathcal{L}_i,$ $\mathcal{N}_i$]
The operators  in \eqref{i.4} and \eqref{4.3} read as
\begin{equation*}
\mathcal{L}_{1}=\sum_{ij=1}^{n}
a_{ij}(x,t)\frac{\partial^{2}}{\partial x_{i}\partial x_{j}}
+\sum_{i=1}^{n}a_{i}(x,t)\frac{\partial }{\partial
x_{i}}+a_{0}(x,t),
\end{equation*}
\begin{equation*}
\mathcal{L}_{2}=\sum_{ij=1}^{n}b_{ij}(x,t)\frac{\partial^{2}}{\partial
x_{i}\partial x_{j}}+\sum_{i=1}^{n}b_{i}(x,t)\frac{\partial
u}{\partial x_{i}}+b_{0}(x,t),
\end{equation*}
and
$$
\mathcal{N}_1=\sum_{i=1}^{n}c_{i}(x,t)\frac{\partial }{\partial
x_{i}}+c_0(x,t),
$$
$$
\mathcal{N}_2=\sum_{i=1}^{n}d_{i}(x,t)\frac{\partial }{\partial
x_{i}}+d_0(x,t).
$$
 There exist positive constants $\delta_i,$
$\delta_2>\delta_1>0$ and $\delta_3>0,$  such that
     \begin{equation}\label{4.5}
\delta_{1}|\xi|^{2}
\leq\sum_{ij=1}^{n}a_{ij}(x,t)\xi_{i}\xi_{j}\leq\delta_{2}|\xi|^{2}
    \end{equation}
   for any $(x,t,\xi)\in\bar{\Omega}_{T}\times \mathbb{R}^{n}$;
    and
    \begin{equation}\label{4.6}
   \sum_{i=1}^{n}c_{i}(x,t)N_{i}(x)  \geq \delta_{3}>0
    \end{equation}
    for any $(x,t)\in\partial\Omega_{T}$,
    where $N=\{N_{1}(x),...,N_{n}(x)\}$ is the unit
    outward normal vector to $\Omega$.
    \smallskip
    \item[h6. Conditions on  $\d_t$] We require that
$$
r_i=\begin{cases} r_i(x,t)\qquad \text{for LO \eqref{i.5}}\\
r_i(t)\qquad\quad \text{for NLO \eqref{i.6}},
\end{cases}
$$
for all $i=0,...,M,$ and
    \[
0<\nu_M<...<\nu_1<\nu_0<1\quad \text{and}\quad \nu_1<\begin{cases}
\frac{\nu_0(1-\alpha)}{2}\qquad\text{in the FDBC case},\\
 \frac{\nu_0(2-\alpha)}{2}\qquad\text{otherwise}.
\end{cases}
    \]
    Moreover, for $i=0,1,...,M,$  there exists a
    positive constant $\delta_4$ such that
    \[
r_0\geq \delta _4>0 \quad \text{for all}\quad
x\in\bar{\Omega}\quad\text{and}\quad t\in[0,T].
    \]
    \item[h7. Smoothness of the coefficients in $\d_t$] For all
    $i=0,1,...,M,$ we assume that
    \[
r_i\in\begin{cases} \C(\bar{\Omega}_T)\qquad\qquad\quad \text{in the
case of
the I type FDO},\\
\C(\bar{\Omega},\C^{1}[0,T])\qquad \text{in the case of the II type
FDO}
\end{cases}
\]
in the case of the local observation, while
\[
r_i\in\begin{cases} \C([0,T])\qquad \text{in the case of
the I type FDO},\\
\C^{1}([0,T])\qquad \text{in the case of the II type FDO},
\end{cases}
\]
in the case of the nonlocal observation.
\item[h8. Regularity of the memory kernels] We require that
\[
\mathcal{K},\, \mathcal{K}_1\in L^{1}(0,T).
\]
 \item[h9. Condition on the additional measurement in \eqref{i.5}, \eqref{i.6}]
    We require that $\psi\in\C([0,t^{*}])$ has $M$-fractional (left)
    regularized Caputo derivatives of order $\nu_i,$ $i=0,1,..,M,$ and all
    these derivatives are H\"{o}lder continuous on $[0,t^*].$

 \item[h10. Regularity of the given functions in \eqref{i.4} and
 \eqref{4.1}] For all $i,j=1,2,...,n,$ we require that
 \[
a_{ij},b_{ij},a_{i},b_{i},a_0,b_0\in\C(\bar{\Omega}_{T}),\qquad
u_0\in C^{2}(\bar{\Omega}),\quad g_0\in\C(\bar{\Omega}_{T}),\quad
g\in\C(\R).
 \]
\end{description}

Assumption \eqref{4.6} on the coefficients $c_i$ tells us that the
vector $c=\{c_1(x,t),...,c_n(x,t)\}$ does not lie in the tangent
plane to $\partial\Omega$ at any point.

Next, accounting assumptions h6-h10, we introduce the quantity
\begin{equation}\label{4.7}
\c_0=
\begin{cases}
g(x_0,0)-g(u_0)|_{x_0}+\mathcal{L}_1u_{0}|_{(x_0,0)}\qquad\qquad\qquad
\text{for
LO \eqref{i.5},}\\
\int\limits_{\Omega}\mathcal{L}_1u_{0}|_{t=0}dx-\int\limits_{\Omega}
g(u_0)dx+\int\limits_{\Omega} g_0(x,0)dx\qquad \text{for NLO
\eqref{i.6}},
\end{cases}
\end{equation}
and state the first main result related to recovery of $\nu_0$ in
\eqref{i.4}.
\begin{theorem}\label{t4.1}
Let for arbitrary finite $T>0$ assumptions h5-h10 hold and $\c_0\neq
0$. We assume that a unique solution $u$ of \eqref{i.4} in
$\Omega_{T}$ is a local classical solution of  \eqref{i.4},
\eqref{4.1} in $\Omega_{t*}$  and, besides, $u$ satisfies
\eqref{i.5} (for the LO) or \eqref{i.6} (for the NLO) in $[0,t^*]$.
Then unknown order $\nu_0$ is computed via the formula
\begin{equation}\label{4.8}
\nu_0=
\begin{cases}
\underset{t\to
0}{\lim}\frac{t[\psi(t)-\psi(0)]}{\int_{0}^{t}[\psi(\tau)-\psi(0)]d\tau}-1\qquad\qquad\quad
\text{for the I type FDO},\\
\\
\underset{t\to
0}{\lim}\frac{t[r_0(t)\psi(t)-r_0(0)\psi(0)]}{\int_{0}^{t}[r_{0}(\tau)\psi(\tau)-r_0(0)\psi(0)]d\tau}-1\qquad
\text{for the II type FDO}
\end{cases}
\end{equation}
with
\[
\psi(0)=\begin{cases} u_0(x_0)\qquad\qquad\text{for LO \eqref{i.5},}\\
\int_{\Omega}u_0(x)dx\qquad \text{for NLO \eqref{i.6}}.
\end{cases}
\]
Moreover, the reconstruction of $\nu_0$ via the observation
$\psi(t),$ $t\in[0,t^*],$ satisfying h9 is unique.
\end{theorem}
It is worth noting that the regularity of the given functions in
Theorem \ref{t4.1} is not sufficient for the existence of a local
classical solution for the problem \eqref{i.4}, \eqref{4.1}. Thus,
despite more general (in applicability) character of Theorem
\ref{t4.1}, this claim is conditional and, in fact, the assumption
of a local classical solvability in \eqref{i.4}, \eqref{4.1} is am
implicit requirement on the given data in the model. However, this
drawback is avoided in our next results via stronger conditions on
the given data which guarantee the unique classical (global)
solvability of \eqref{i.4}, \eqref{4.1}-\eqref{4.4}. To this end, we
need additional hypothesises, which differ slightly for linear and
nonlinear versions of the equation \eqref{i.4}.
\begin{description}

 \item[h11. Regularity of the right-hand side in  \eqref{i.4} and
 \eqref{4.1}-\eqref{4.4}] We require that
 \[
u_0\in\C^{2+\alpha}(\bar{\Omega}),\quad
g_0\in\C^{\alpha,\alpha/2}(\bar{\Omega}_T)\quad\text{and}\quad
\varphi_1,\varphi_2\in\C^{1+\alpha,\frac{1+\alpha}{2}}(\partial\Omega_{T}).
 \]
 \item[h12. Higher regularity of the coefficients in the operators involved] For $i,j=1,...,n,$
 \[
a_{ij},b_{ij},a_i,b_i,a_0,b_0\in\C^{\alpha,\frac{\alpha}{2}}(\bar{\Omega}_T),
 \]
 and
 \[
c_i,c_0,d_i,d_0\in\C^{1+\alpha,\frac{1+\alpha}{2}}(\partial\Omega_T).
 \]
 Moreover, for $k=0,1,...,M,$ we assume that
 \[
r_i\in\begin{cases}
\C^{\alpha,\frac{\alpha}{2}}(\bar{\Omega}_{T})\qquad\qquad\qquad\qquad\text{in
the DBC
or 3BC cases,}\\
\C^{\alpha,\frac{\alpha}{2}}(\bar{\Omega}_{T})\cap\C^{1+\alpha,\frac{1+\alpha}{2}}(\partial\Omega_{T})\qquad\text{in
the FDBC case}
\end{cases}
 \]
in the case of the I
 type FDO,
while
\[
r_i\in\begin{cases} \C^{1+\frac{\alpha}{2}}([0,T],
\C^1(\bar{\Omega}))\qquad\qquad\qquad\qquad\qquad\text{in the DBC
or 3BC cases,}\\
\C^{1+\frac{\alpha}{2}}([0,T],
\C^1(\bar{\Omega}))\cap\C^{\frac{3+\alpha}{2}}([0,T],\C^{2}(\partial\Omega))\qquad\text{in
the FDBC case}
\end{cases}
 \]
 in the case of the II type FDO.
 \item[h13. Compatibility conditions] The following compatibility conditions hold for every
         $x\in\partial\Omega$ at the initial time
     $t=0$:
    \begin{equation*}
0=u_{0}(x),\quad 0=\mathcal{L}_{1}u_{0}(x)|_{t=0}+g_0(x,0),
    \end{equation*}
    in the \textbf{DBC} case; and
    \begin{equation*}
\mathcal{N}_{1}u_{0}(x)|_{t=0}=\varphi_{1}(x,0),
    \end{equation*}
    in the \textbf{3BC}  case; and
            \begin{equation*}
\mathcal{L}_{1}u_{0}(x)|_{t=0}+g_0(x,0)=\varphi_{2}(x,0)+\mathcal{N}_{1}u_{0}(x)|_{t=0},
    \end{equation*}
    if \textbf{FDBC} takes place.
\end{description}
\begin{theorem}\label{t4.2}
Let arbitrary positive $T$ be finite, and $g(u)\equiv 0,$
assumptions h5, h6, h8, h9, h11-h13 hold. Assume, additionally,
$\c_0\neq 0.$ Then the inverse problem \eqref{i.1}, \eqref{i.4},
\eqref{i.5}, \eqref{4.1}-\eqref{4.4} has a unique solution
$(\nu_0,u),$ where $\nu_0$ is computed via \eqref{4.8} and $u$ is a
unique classical (global) solution of the direct problem
\eqref{i.1}, \eqref{i.4}, \eqref{4.1}-\eqref{4.4} having the
regularity
\[
u\in\C^{2+\alpha,\frac{2+\alpha}{2}\nu_0}(\bar{\Omega}_{T})\quad\text{and}\quad
\d_t^{\nu_i}u\in\C^{\alpha,\frac{\alpha\nu_0}{2}}(\bar{\Omega}_{T}),\quad
i=1,...,M,
\]
and, additionally, in the case of FDBC,
$\d_t^{\nu_i}u\in\C^{1+\alpha,\frac{1+\alpha}{2}\nu_0}(\partial\Omega_T)$.
\end{theorem}
The very similar result holds in the case of the nonlocal
observation \eqref{i.6}.
\begin{theorem}\label{t4.3}
Let arbitrary positive $T$ be finite, and $g(u)\equiv 0,$
assumptions h5, h6, h8, h9, h11-h13 hold. Assume additionally
$\c_0\neq 0.$ Then the inverse problem \eqref{i.1}, \eqref{i.4},
\eqref{i.6}, \eqref{4.1}-\eqref{4.4} has a unique solution
$(\nu_0,u),$ where $\nu_0$ is computed via \eqref{4.8} and $u$ is a
unique classical (global) solution of the direct problem
\eqref{i.1}, \eqref{i.4}, \eqref{4.1}-\eqref{4.4} having the
regularity stated in Theorem \ref{t4.2}. Moreover, for any positive
$T_0\in(0,T]$, there holds
\[
\int_{\Omega}udx\in\C^{\nu_0}([0,T_0])\quad\text{and}\quad
\d_t^{\nu_i}\int_{\Omega}udx\in\C^{\frac{\alpha\nu_0}{2}}(\bar{\Omega}_{T}),\quad
i=0,1,...,M,\quad \d_t \int_{\Omega}udx|_{t=0}=\c_0.
\]
\end{theorem}
\begin{remark}\label{r4.0}
In fact, assumptions on the smoothness of the coefficients and the
right-hand sides in \eqref{i.4}, \eqref{4.2} and \eqref{4.3} are a
little more redundant than it is necessary for the global classical
solvability of the corresponding direct problems. Namely, for
example, we need only
$\varphi_i\in\C^{1+\alpha,\frac{1+\alpha}{2}\nu_0}(\partial\Omega_T)$,
$g_0\in\C^{\alpha,\alpha\nu_0/2}(\bar{\Omega}_T)$. However, since
parameter $\nu_0$ is unknown and should be found via solving of the
corresponding inverse problems, we cannot examine the belonging of
these functions to the space
$\C^{l+\alpha,\frac{l+\alpha}{2}{\nu_0}},$ $l=0,1,$ and by this
reason, we require more regularity of the given functions.
\end{remark}

We complete this section with the last but not least result
concerning with the semilinear equation \eqref{i.4}. To this end, we
state additional conditions on the nonlinearity and the coefficients
in $\d_t$ given by \eqref{i.1}. Here, we analyze two types of the
nonlinearity $g(u)$.
\begin{description}
 \item[h14. Conditions on the nonlinear term] We assume that $g$
 meets  the one of the following requirements:
 \item[h14.(i)] either $g(u)$ is the local Lipschitz and admits no
 more than linear growth, i.e. for every positive quantity $\mathrm{a}$,
 there exists a positive constant $C_{\mathrm{a}}$ such that
\[
|g(u_1)-g(u_2)|\leq C_{\mathrm{a}}|u_1-u_2|
\]
for each $u_1,u_2\in[-\mathrm{a},\mathrm{a}]$; and there is a
positive constant $\mathrm{A}$ such that
\[
|g(u)|\leq \mathrm{A}(1+|u|)\quad\text{for all}\quad u\in\R ;\]

 \item[h14.(ii)] or $g\in\C^{1}(\R)$ and for some nonnegative
 constants $\mathrm{A}_{i},$ $i=1,2,3,4,$ and $\mathrm{b}\geq 0$,
 the inequalities hold
 \[
\begin{cases}
|g(u)|\leq\mathrm{A}_1(1+|u|^{\mathrm{b}}),\\
ug(u)\geq-\mathrm{A}_2+\mathrm{A}_3|u|^{\mathrm{b}+1},\\
g'(u)\geq -\mathrm{A}_4.
\end{cases}
 \]
\end{description}
\begin{example}\label{e4.1}
The following is an example of the nonlinearity $g$ satisfying
conditions h14.(i) and h14.(ii), respectively:
\[
g(u)=\frac{\sin\,u}{2+u^2}\quad\text{and}\quad
g(u)=\sum_{i=1}^{2m}a_{2m-i}u^{i-1}
\]
with $m\in\mathbb{N}$ and $a_i$ being real constants. We notice that
the nonlinearity in the form of an odd degree polynomial with the
positive leading coefficient is utilized to model the heat flow in a
rigid, isotropic, homogenous heat conductor with memory
\cite[Section 1]{GPM}.
\end{example}
In the case of the semilinear equation \eqref{i.4}, we examine the
initial problem \eqref{i.1}, \eqref{4.1} subject to either the
Dirichlet boundary condition \eqref{4.2} or  the boundary condition
of the third kind \eqref{4.3}.

We recall that the classical solvability of the direct problems
\eqref{4.1}-\eqref{4.4} for linear version of \eqref{i.4}
($g(u)\equiv 0$) does not require any restrictions on the sign of
the coefficients $r_i,$ $i=1,...,M,$ in $\d_t$ given by \eqref{i.1}
(see \cite{PSV}). However,  the classical solvability of these
problems for semilinear \eqref{i.4} (contrary to the linear case)
needs  additional requirements on the sign and the behavior of the
coefficients (see for more details \cite{SV,V1,V2}). In order to
state them here, we rewrite $\d_t$ in a bit different form than
\eqref{i.1},
\begin{equation}\label{4.9}
\d_t=
\begin{cases}
r_0\d_t^{\nu_0}+\sum\limits_{i=1}^{M_1}r_i\d_t^{\nu_i}-\sum\limits_{j=1}^{M_2}\gamma_j\d_t^{\bar{\nu}_j},
\quad\text{the I type FDO},\\
\d_t^{\nu_0}r_0+\sum\limits_{i=1}^{M_1}\d_t^{\nu_i}r_i-\sum\limits_{j=1}^{M_2}\d_t^{\bar{\nu}_j}\gamma_j,
\quad\text{the II type FDO},
\end{cases}
\end{equation}
where $\nu_i,\bar{\nu}_j\in(0,\nu_0),$ $\nu_i\neq\bar{\nu}_{j}$,
$i=1,...,M_1,$ $j=1,...,M_2,$ are any but fixed, and
$\gamma_i=\gamma_i(x,t)$ are given coefficients. Here, we agreed to
establish that for $M_k=0,$ $k=1,2,$ the corresponding sum in
\eqref{4.9} is zero.

Moreover, instead of $\mathcal{L}_i,$ $\mathcal{N}_i$ stated in h5,
we consider the operators in the forms:
\begin{align}\label{4.10}\notag
\mathcal{L}_1&=\sum_{ij=1}^{n}\frac{\partial}{\partial
x_i}a_{ij}(x,t)\frac{\partial}{\partial
x_i}+\sum_{i=1}^{n}a_i(x,t)\frac{\partial}{\partial x_i}+a_0(x,t),\\
\notag \mathcal{L}_2&=\sum_{ij=1}^{n}\frac{\partial}{\partial
x_i}a_{ij}(x,t)\frac{\partial}{\partial
x_i}+\sum_{i=1}^{n}b_i(x,t)\frac{\partial}{\partial x_i}+b_0(x,t),\\
\mathcal{N}_1&=\mathcal{N}_2=-\sum_{ij=1}a_{ij}(x,t)N_i\frac{\partial}{\partial
x_i}-c_0,
\end{align}
where $c_0$ is a positive constant.

At this point, we state additional conditions on the operators
described in \eqref{4.9} and \eqref{4.10}.
\begin{description}
 \item[h15. Conditions on the coefficients in \eqref{4.10}] We
 assume that for $i,j=1,...,n,$ the ellipticity condition \eqref{4.5} holds and
 $a_0,b_0$ meet the requirements stated in h12, while
 \[
a_i,b_i,a_{ij}\in C^{1+\alpha,\frac{1+\alpha}{2}}(\bar{\Omega}_T).
 \]
 \item[h16. Assumptions on the orders $\nu_i$ and $\bar{\nu}_i$ in
 \eqref{4.9}] We require
 \[
0<\bar{\nu}_{M_2}<...<\bar{\nu}_1<\nu_0<1,\qquad
0<\nu_{M_1}<...<\nu_1<\nu_0<1,
 \]
 and
 \[
\max\{\nu_1,\bar{\nu_1}\}<\frac{\nu_0(2-\alpha)}{2}.
 \]
  \item[h17. Assumptions on the coefficients  in
 \eqref{4.9}] We require that for some $\alpha^*\in(1,1+\alpha/2]$,
 $r_i,\gamma_i\in\C^{\alpha^*}([0,T],\C^{1}(\bar{\Omega}))$ are
 nondecreasing  with respect to time are strictly positive in
 $\bar{\Omega}_{T}$. Besides, there exists a strictly positive,
 nondecreasing in time  function $R=R(x,t)$ belonging to $\C^{\alpha^*}([0,T],\C^{1}(\bar{\Omega}))$ such that
 \[
r_0(x,t)=R(x,t)+\sum_{j=1}^{M_2}\gamma_{j}(x,t).
 \]
In the case of local observation \eqref{i.6}, we, additionally,
assume that the coefficients in $\d_t$ are only time dependent, that
is $r_i=r_i(t)$ and $\gamma_i=\gamma_i(t)$.
 \end{description}
The next results concerns with the local observation \eqref{i.5}.
\begin{theorem}\label{t4.4}
Let positive $T$ be arbitrary finite and let  h9, h11, h15--h17
hold. We assume that the consistency condition h13 with
$g_{0}(x,0)-g(u_0)$ in place of $g_{0}(x,0)$ holds,
$\mathcal{K}\in\C^{1}([0,T])$ and the nonlinear term meets h14 if
$M_2=0$ and $h14.(i)$, if $M_2\geq 1$. If $\c_0\neq 0$, then inverse
problem \eqref{i.4}, \eqref{i.5}, \eqref{4.1}-\eqref{4.3} with the
operators given by \eqref{4.9} and \eqref{4.10} has a unique
solution $(\nu_0,u)$, where $\nu_0$ is computed via \eqref{4.8} and
$u$ is a unique classical (global) solution of the problem
\eqref{i.4}, \eqref{4.1}-\eqref{4.3} and, besides, $u$ has the
regularity stated in Theorem \ref{t4.2}.
\end{theorem}
Coming to the nonlocal observation \eqref{i.6}, we claim.
\begin{theorem}\label{t4.5}
Under assumptions of Theorem \ref{t4.4}, the inverse problem
\eqref{i.4}, \eqref{i.6}, \eqref{4.1}-\eqref{4.3}, \eqref{4.9} and
\eqref{4.10} admits a unique solution $(\nu_0,u)$ with $\nu_0$ being
computed via \eqref{4.8} and $u$ being a unique classical (global)
solution of the problem \eqref{i.4}, \eqref{4.1}-\eqref{4.3} having
regularity established in Theorem \ref{t4.3}.
\end{theorem}
\begin{remark}\label{r4.2}
Due to \cite{SV}, in the one-dimensional case, assumptions on  the
memory kernel and on the operators $\mathcal{L}_i,$
$\mathcal{N}_{i}$ in semilinear \eqref{i.4} may be relaxed.
 In fact, if the
coefficients of $\d_t$ are only time-dependent, then it is enough to
require that $\mathcal{K}\in L^{1}(0,T)$, $\mathcal{L}_i$ and
$\mathcal{N}_i$ satisfy h5 and h12.
\end{remark}

The proofs of Theorems \ref{t4.1}-\ref{t4.5} are given in Section
\ref{s6}.

\section{Proof of Theorems \ref{t3.1}-\ref{t3.2}}\label{s5}

\noindent The main technical tool in the proof is Lemma 4.1,
established in \cite{KPSV1}, and which, for reader's convenience, is
reported here  in a particular form tailored for our aims.
\begin{lemma}\label{l5.1}
Let an arbitrary positive $T$ be finite and $\beta\in(0,1)$. Assume
that a continuous function $\phi=\phi(t):[0,T]\to\R$ has a
continuous fractional Caputo derivative $\d_t^{\beta}\phi(t)$ on
$[0,T],$ then the representation
\[
\phi(t)=\phi(0)+\frac{\d_{t}^{\beta}\phi(0)}{\Gamma(1+\beta)}t^{\beta}+\frac{1}{\Gamma(\beta)}\int\limits_{0}^{t}(t-\tau)^{\beta-1}[\d_t^{\beta}\phi(\tau)
-\d_{t}^{\beta}\phi(0)]d\tau
\]
 holds for all $t\in[0,T]$.

If, in addition,
\begin{equation}\label{5.1}
\d_{t}^{\beta}\phi(0)\neq 0\qquad\text{and}\qquad
\phi\in\C_{\beta}^{\beta^*}([0,T^*])
\end{equation}
for some $\beta^*\in(0,1)$ and some $T^*\leq T$,  then
\begin{equation}\label{5.2}
\beta+1=\underset{t\to
0}{\lim}\frac{t[\phi(t)-\phi(0)]}{\int_{0}^{t}[\phi(\tau)-\phi(0)]d\tau}.
\end{equation}
\end{lemma}
Indeed, applying  this lemma to the function
\[
\phi=\begin{cases} v\qquad\text{for the I type FDO,}\\
r_0v \quad\text{for the II type FDO,}
\end{cases}
\]
and selecting $\beta=\nu_0$ complete the proof of Theorems
\ref{t3.1} and \ref{t3.2} only if we demonstrate that  $v$ (in the
case of the I type FDO) and $r_0 v$ (in the case of the II type FDO)
meet requirements of Lemma \ref{l5.1}. On this route, the main
difficulty concerns with the verification of the relations
\[
\d_{t}^{\nu_0}v|_{t=0}\neq 0 \quad\text{for the I type FDO} \quad
\text{and}\quad \d_{t}^{\nu_0}(r_0v)|_{t=0}\neq 0 \quad\text{for the
II type FDO}.
\]
Moreover, in the case of the II type FDO, we have to verify that
$\d_{t}^{\nu_0}(r_0v)\in\C^{\nu^*}([0,T^*])$ for some
$\nu^*\in(0,1),$ and $r_0v$ satisfies \eqref{5.1}. To achieve the
desired results, we need the following key lemmas stated in the
following subsection.


\subsection{Key Lemmas}\label{s5.1}
\begin{lemma}\label{l5.2}
Let positive $T$ be given, $0<\beta_2<\beta_1<1$ and let the
continuous function $\Phi:=\Phi(t):[0,T]\to\R$ have the continuous
derivatives $\d_t^{\beta_1}\Phi$, $\d_t^{\beta_2}\Phi$ on $[0,T]$.
Assume  that for some $\beta_3\in(0,1-\beta_1),$  there exists
positive $T^*\leq T$ such that
$\d_t^{\beta_i}\Phi\in\C^{\beta_3}([0,T^*]),$ $i=1,2.$ Then
\begin{equation}\label{5.3}
\d_t^{\beta_2}\Phi|_{t=0}=0.
\end{equation}
\end{lemma}
\begin{proof}
Obviously, if $\Phi\equiv const.,$ then the desired equality follows
immediately from the definition of the Caputo derivative. Thus,
assuming that $\Phi\not\equiv const$ and appealing to properties of
the function $\Phi$, we can apply \cite[Theorem 3.1]{KPSV}, which
gives the following equivalent definition for the Caputo derivative
\begin{equation}\label{5.4}
\d_t^{\beta_i}\Phi(t)=\frac{\Phi(t)-\Phi(0)}{t^{\beta_i}\Gamma(1-\beta_i)}+\frac{\beta_i}{\Gamma(1-\beta_i)}\int_0^{t}
\frac{\Phi(t)-\Phi(\tau)}{(t-\tau)^{1+\beta_i}}d\tau,\quad i=1,2.
\end{equation}
At this point, exploiting \eqref{5.4}, we aim to obtain the estimate
\begin{equation}\label{5.5}
|\d_t^{\beta_2}\Phi(t)|\leq
\Phi_0(t)\|\Phi\|_{\C_{\beta_1}^{\beta_3}([0,T^*])}
\end{equation}
for all $t\in[0,T^*]$ with the  continuous function $\Phi_0(t)$
being positive for $t>0$ and vanishing at $t=0$. Then, passing to
the limit in this inequality as $t\to0$, we arrive at the desired
estimate \eqref{5.3} and, accordingly, complete the proof of this
lemma. To active \eqref{5.5}, we evaluate each term in the
right-hand side of \eqref{5.4} with $i=2,$ separately.

Appealing to the regularity of the function $\Phi$ and exploiting
the representation in Lemma \ref{l5.1} yield
\begin{align}\label{5.6}\notag
\frac{\Phi(t)-\Phi(0)}{t^{\beta_2}}&\leq
t^{\beta_1-\beta_2}\Big[\frac{|\d_t^{\beta_1}\Phi(0)|}{\Gamma(1+\beta_1)}
+\frac{t^{-\beta_1}\langle\d_t^{\beta_1}\Phi\rangle^{(\beta_3)}_{t,[0,T^*]}}{\Gamma(\beta_1)}\int_0^{t}(t-\tau)^{\beta_1-1}\tau^{\beta_3}d\tau
\Big]\\\notag & =t^{\beta_1-\beta_2}\Big[
\frac{|\d_t^{\beta_1}\Phi(0)|}{\Gamma(1+\beta_1)}+\frac{t^{\beta_3}\Gamma(1+\beta_3)}
{\Gamma(1+\beta_3+\beta_1)}\langle\d_t^{\beta_1}\Phi\rangle_{t,[0,T^*]}^{(\beta_3)}
\Big]\\& \leq \Phi_1(t)\|\Phi\|_{\C_{\beta_1}^{\beta_3}([0,T^*])},
\end{align}
where we set
\[
\Phi_1(t)=t^{\beta_1-\beta_2}\Big[\frac{1}{\Gamma(1+\beta_1)}+\frac{t^{\beta_3}\Gamma(1+\beta_3)}{\Gamma(1+\beta_3+\beta_1)}\Big].
\]

Coming to the second term in \eqref{5.4} with $i=2$, we rewrite the
difference $\Phi(t)-\Phi(\tau)$ in the form
\begin{align*}
\Phi(t)-\Phi(\tau)&=\frac{\d_t^{\beta_1}\Phi(0)}{\Gamma(1+\beta_1)}[t^{\beta_1-\tau^{\beta_1}}]+
I_t^{\beta_1}[\d_t^{\beta_1}\Phi-\d_t^{\beta_1}\Phi(0)](t)\\
&- I_t^{\beta_1}[\d_t^{\beta_1}\Phi-\d_t^{\beta_1}\Phi(0)](\tau).
\end{align*}
Then, taking into account that $0<\tau<t$, we have the chain of
inequalities
\begin{align*}
|\Phi(t)-\Phi(\tau)|&\leq
\frac{\d_t^{\beta_1}\Phi(0)}{\Gamma(1+\beta_1)}[t^{\beta_1}-\tau^{\beta_1}]+[t-\tau]^{\beta_1+\beta_3}
\|I_t^{\beta_1}[\d_t^{\beta_1}\Phi-\d_t^{\beta_1}\Phi(0)]\|_{\C^{\beta_1+\beta_3}([0,T^*])}\\
& \leq
\frac{\d_t^{\beta_1}\Phi(0)}{\Gamma(1+\beta_1)}[t^{\beta_1}-\tau^{\beta_1}]+[t-\tau]^{\beta_1+\beta_3}\|\Phi\|_{\C_{\beta_1}^{\beta_3}([0,T^*])}.
\end{align*}
Here, to arrive at the last inequality, we used \eqref{2.1}.

At last, employing the last estimate, we obtain
\begin{align}\label{5.7}\notag
\frac{\beta_2}{\Gamma(1-\beta_2)}\Big|\int_0^{t}\frac{\Phi(t)-\Phi(\tau)}{(t-\tau)^{1+\beta_2}}d\tau\Big|&\leq
\frac{\beta_2}{\Gamma(1-\beta_2)}\frac{|\d_t^{\beta_1}\Phi(0)|}{\Gamma(1+\beta_1)}\int_0^{t}\frac{t^{\beta_1}-\tau^{\beta_1}}{(t-\tau)^{1+\beta_2}}d\tau
\\ \notag
&
+\frac{\beta_2}{\Gamma(1-\beta_2)}\|\Phi\|_{\C_{\beta_1}^{\beta_3}([0,T^*])}\int_0^{t}(t-\tau)^{\beta_1+\beta_3-\beta_2-1}d\tau\\\notag
&\leq \frac{\beta_2t^{\beta_1-\beta_2}}{\Gamma(1-\beta_2)}\Big[
\frac{|\d_t^{\beta_1}\Phi(0)|}{\Gamma(1+\beta_1)}\int_{0}^{t}\frac{t^{\beta_1}-\tau^{\beta_1}}{(t-\tau)^{1+\beta_1}}d\tau
+
\frac{t^{\beta_3}\|\Phi\|_{\C_{\beta_1}^{\beta_3}([0,T^*])}}{\beta_3+\beta_1-\beta_2}
 \Big]\\
&\leq \Phi_2(t)\|\Phi\|_{\C_{\beta_1}^{\beta_3}([0,T^*])}
\end{align}
for all $t\in[0,T^*]$, where we set
\[
\Phi_2(t)=\frac{\beta_2 t^{\beta_1-\beta_2}}{\Gamma(1-\beta_2)}\Big[
\frac{\pi\beta_1}{\Gamma(1+\beta_1)\sin(\pi\beta_1)}
+\frac{t^{\beta_3}}{\beta_1+\beta_3-\beta_2}\Big].
\]
It is worth noting that the integral in the last inequalities is
computed via formula in (iii) \cite[Proposition 5.1]{KPSV}. Finally,
collecting estimate \eqref{5.6} with \eqref{5.7}, we arrive at the
bound \eqref{5.5} with
\[
\Phi_0(t)=\Phi_1(t)+\Phi_2(t).
\]
Then, bearing in mind vanishing $\Phi_i(t)$ at $t=0$, we complete
the proof of this claim.
\end{proof}
Our next assertion being a key point in the proof of Theorem
\ref{t3.2} describes the regularity of the product
$\phi_1(t)\phi_2(t)$, where each function
$\phi_i:=\phi_i(t):[0,T]\to\R$ has a fractional Caputo derivative of
order less than $1$.
\begin{lemma}\label{l5.3}
Let $\beta_0,\beta_1\in(0,1)$ and let
$\phi_1\in\C_{\beta_1}^{\beta_0}([0,T])$.

\noindent(i) If $\beta_0+\beta_1<1$ and
$\phi_2\in\C^{\beta_2}([0,T])\cap\C_{\beta_1}^{\beta}([0,T])$ with
some $\beta_2\in(\beta_0+\beta_1,1),$ $\beta\in(0,1)$, then
$\phi_1\phi_2\in\C^{\beta_3}_{\beta_1}([0,T])$ with
$\beta_3=\min\{\beta,\beta_0,\beta_2-\beta_1\}$ and, besides,
\[
\|\phi_1\phi_2\|_{\C_{\beta_1}^{\beta_3}([0,T])}\leq
C[\|\phi_2\|_{\C_{\beta_1}^{\beta}([0,T])}+\|\phi_2\|_{\C^{\beta_2}([0,T])}]\|\phi_1\|_{\C_{\beta_1}^{\beta_0}([0,T])}.
\]

\noindent(ii) If $\phi_2\in\C^{\beta_2}([0,T])$ with some
$\beta_2>\beta_0+\beta_1\geq 1$, then
$\phi_1\phi_2\in\C^{\beta_3}_{\beta_1}([0,T])$ with
$\beta_3=\min\{\beta,\beta_0\}$ and, besides,
\[
\|\phi_1\phi_2\|_{\C_{\beta_1}^{\beta_3}([0,T])}\leq
C\|\phi_2\|_{\C^{\beta_2}([0,T])}\|\phi_1\|_{\C_{\beta_1}^{\beta_0}([0,T])}.
\]

\noindent(iii) If for some $\beta_2\in(\beta_1,1)$ and
$\beta,\beta_4\in(0,1)$ such that $\beta+\beta_2<1$ and
$\beta_0+\beta_1<1,$
 $\phi_2\in\C_{\beta_2}^{\beta}([0,T])$ and
$\d_t^{\beta_1}\phi_2\in\C^{\beta_4}([0,T])$,    then
$\phi_1\phi_2\in\C_{\beta_1}^{\beta_3}([0,T])$ with
$\beta_3=\min\{\beta_0,\beta_1,\beta,\beta_4,\beta_2-\beta_1\}$ and,
besides,
\[
\|\phi_1\phi_2\|_{\C_{\beta_1}^{\beta_3}([0,T])}\leq
C[\|\phi_2\|_{\C_{\beta_2}^{\beta}([0,T])}+\|\d_t^{\beta_1}\phi_2\|_{\C^{\beta_4}([0,T])}]\|\phi_1\|_{\C_{\beta_1}^{\beta_0}([0,T])}.
\]
The positive quantity $C$ depends only on $T,\beta_i,\beta$.
\end{lemma}
\begin{proof}
The proof of statements in (i) and (ii) repeats the arguments
leading to (iii) in \cite[Lemma 5.1]{PSV} (see also Remark 5.4
therein) and \cite[Proposition 5.5, Lemma 5.6]{SV}, respectively.
Thus, we are left to verify (iii) in this claim. Taking into account
the regularity of $\phi_i$ and Definition \ref{d2.1}, we will end up
with the results stated in (iii), if we obtain the embedding
$\d_t^{\beta_1}(\phi_1\phi_2)\in\C^{\beta_3}([0,T])$. To this end,
exploiting the "fractional" Leibniz rule established in
\cite[Corollary 2]{V2}, we compute the Caputo  derivative of the
product $(\phi_1\phi_2)$ as
\begin{equation*}\label{5.8}
\d_t^{\beta_1}(\phi_1\phi_2)(t)=\phi_2(t)\d_{t}^{\beta_1}\phi_1(t)+\phi_1(0)\d_{t}^{\beta_1}\phi_2(t)
+\frac{\beta_1}{\Gamma(1-\beta_1)}J_{\beta_1}(t;\phi_2,\phi_1),
\end{equation*}
where
\[
J_{\beta_1}(t)=J_{\beta_1}(t;\phi_2,\phi_1)=\int_0^{t}\frac{[\phi_2(t)-\phi_2(\tau)][\phi_1(\tau)-\phi_1(0)]}{(t-\tau)^{1+\beta_1}}d\tau.
\]
Clearly, the first two terms in the right-hand side in the
representation to $\d_t^{\beta_1}(\phi_1\phi_2)(t)$
 belong to $\C^{\beta_3}([0,T])$ and
\begin{equation*}\label{5.9}
\|\phi_2(t)\d_{t}^{\beta_1}\phi_1(t)+\phi_1(0)\d_{t}^{\beta_1}\phi_2(t)\|_{\C^{\beta_3}([0,T])}\leq
C[\|\phi_2\|_{\C_{\beta_2}^{\beta}([0,T])}+\|\d_t^{\beta_1}\phi_2\|_{\C^{\beta_4}([0,T])}]\|\phi_1\|_{\C_{\beta_1}^{\beta_0}([0,T])}
\end{equation*}
with $C$ depending only on $T,\beta_i$ and $\beta$.

Coming to the estimate of $J_{\beta_1}(t),$ we exploit  the
representation of $\phi_2$ and rewrite the difference
$\phi_2(t_2)-\phi_2(t_1)$ in the form
\begin{equation}\label{5.10}
\phi_2(t_2)-\phi_2(t_1)=I_{t}^{\beta_2}[\d_t^{\beta_2}\phi_2-\d_t^{\beta_2}\phi_2(0)](t_2)-
I_{t}^{\beta_2}[\d_t^{\beta_2}\phi_2-\d_t^{\beta_2}\phi_2(0)](t_1) +
\frac{\d_t^{\beta_2}\phi_2(0)}{\Gamma(1+\beta_2)}[t_2^{\beta_2}-t^{\beta_2}_{1}]
\end{equation}
which (see \eqref{2.1}) in turn leads to the bound
\begin{equation*}\label{5.11}
|\phi_2(t_2)-\phi_2(t_1)|\leq
C[|t_2-t_1|^{\beta_2+\beta}\|\phi_2\|_{\C_{\beta_2}^{\beta}}([0,T])+|\d_t^{\beta_2}\phi_2(0)||t_2-t_1|^{\beta_2}]
\end{equation*}
with $C$ depending only  on $T,\beta_2$ and $\beta$.

Thus, exploiting this estimate (with $t_2=t$ and $t_1=\tau$) and
appealing to the smoothness of $\phi_1$ yield the bound
\begin{align*}
|J_{\beta_1}(t)|&\leq
C\|\d_t^{\beta_1}\phi_1\|_{\C([0,T])}[\|\phi_2\|_{\C_{\beta_2}^{\beta}([0,T])}+\|\d_t^{\beta_2}\phi_2\|_{\C([0,T])}]\\
& \times \Big[
\int_0^{t}(t-\tau)^{\beta_2+\beta-\beta_1-1}\tau^{\beta_1}d\tau +
\int_{0}^{t}(t-\tau)^{\beta_2-\beta_1-1}\tau^{\beta_1}d\tau \Big].
\end{align*}
Then, performing the change of variable $z=\tau/t$ and appealing to
the definition of $B-$function, we obtain
\begin{align}\label{5.12}\notag
|J_{\beta_1}(t)|&\leq
C\|\d_t^{\beta_1}\phi_1\|_{\C([0,T])}[\|\phi_2\|_{\C_{\beta_2}^{\beta}([0,T])}+\|\d_t^{\beta_2}\phi_2\|_{\C([0,T])}]\\&\times
[t^{\beta_2+\beta}B(\beta_2+\beta-\beta_1,\beta_1+1)+t^{\beta_2}B(\beta_2-\beta_1,\beta_1+1)].
\end{align}
We recall that
$$B(\theta_1,\theta_2)=\frac{\Gamma(\theta_1)\Gamma(\theta_2)}{\Gamma(\theta_1+\theta_2)}=\int_{0}^{1}z^{\theta_1-1}(1-z)^{\theta_2-1}dz\quad
\text{with}\quad Re\, \theta_i>0.$$

At this point, we obtain the bound of the seminorm $\langle
J_{\beta_1}\rangle_{t,[0,T]}^{(\beta_3)}$. To this end, exploiting
equality \eqref{5.10}, we rewrite $J_{\beta_1}$ in more suitable
(for the further study) form
\begin{equation*}\label{5.12*}
J_{\beta_1}(t)=I_1(t)+I_2(t),
\end{equation*}
where
\begin{align*}
I_1(t)&=\frac{\d_t^{\beta_2}\phi_2(0)}{\Gamma(1+\beta_2)}\int_0^{t}\frac{t^{\beta_2}-(t-\tau)^{\beta_2}}{\tau^{\beta_1+1}}I_t^{\beta_1}\d_t^{\beta_1}\phi_1(t-\tau)d\tau,\\
I_2(t)&=\int_0^{t}[I^{\beta_2}_t(\d_t^{\beta_2}\phi_2-\d_t^{\beta_2}\phi_2(0))(t)
-I^{\beta_2}_t(\d_t^{\beta_2}\phi_2-\d_t^{\beta_2}\phi_2(0))(t-\tau)]\\
& \times
\tau^{-\beta_1-1}I_t^{\beta_1}\d_t^{\beta_1}\phi_1(t-\tau)d\tau.
\end{align*}
Next, letting $0\leq t_1<t_2\leq T$ and setting
$$\Delta t=t_2-t_1,\quad \Delta_t I_1=I_1(t_2)-I_{1}(t_1),\quad  \Delta_t I_2=I_2(t_2)-I_{2}(t_1),$$
we evaluate, separately, each difference $\Delta_t I_i.$

\noindent$\bullet$ As for $\Delta_t I_1$, introducing the new
variable $y=\frac{\tau}{t},$ we have
\begin{align*}
I_1(t)&=\frac{t^{\beta_2-\beta_1}\d_t^{\beta_2}\phi_2(0)}{\Gamma(1+\beta_2)}\int_{0}^{1}\frac{1-(1-y)^{\beta_2}}{y^{\beta_1+1}}I_t^{\beta_1}\d_t^{\beta_1}\phi_1(t(1-y))dy
\\
& =
\frac{t^{\beta_2-\beta_1}\d_t^{\beta_2}\phi_2(0)}{\Gamma(\beta_2)}\int_{0}^{1}
y^{-\beta_1}\int_0^{1}(1-qy)^{\beta_2-1}dq\,
I_t^{\beta_1}\d_t^{\beta_1}\phi_1(t(1-y))dy.
\end{align*}
Bearing in mind this equality and setting
\begin{align*}
I_{1,1}&=[t^{\beta_2-\beta_1}_{2}-t^{\beta_2-\beta_1}_{1}]\int_0^1y^{-\beta_1}I_t^{\beta_1}\d_t^{\beta_1}\phi_1(t_2(1-y))\int_{0}^1(1-qy)^{\beta_2-1}dq
dy,\\
I_{1,2}&=\frac{t^{\beta_2-\beta_1}_{1}}{\Gamma(1+\beta_1)}[t_2^{\beta_1}-t_1^{\beta_1}]\d_t^{\beta_1}\phi_1(0)
\int_{0}^{1}y^{-\beta_1}(1-y)^{\beta_1}\int_{0}^{1}(1-qy)^{\beta_2-1}dqdy,\\
I_{1,3}&=t^{\beta_2-\beta_1}_{1}\int_0^{1}y^{-\beta_1}[I_t^{\beta_1}(\d_t^{\beta_1}\phi_1-\d_t^{\beta_1}\phi_1(0))(t_2(1-y))
-
I_t^{\beta_1}(\d_t^{\beta_1}\phi_1-\d_t^{\beta_1}\phi_1(0))(t_1(1-y))]\\
& \times\int_0^{1}(1-qy)^{\beta_2-1}dqdy,
\end{align*}
we rewrite $\Delta_tI_1$ in the form
\[
\Delta_t
I_1=\frac{\d_t^{\beta_2}\phi_2(0)}{\Gamma(\beta_2)}\sum_{j=1}^{3}I_{1,j}.
\]
Straightforward calculations along with \eqref{2.1} lead to
relations
\begin{align*}
|I_{1,1}|&\leq C(\Delta
t)^{\beta_2-\beta_1}\|\d_t^{\beta_1}\phi_1\|_{\C([0,T])}\int_{0}^{1}y^{-\beta_1}(1-y)^{\beta_1+\beta_2-1}dy\\
& \leq C(\Delta
t)^{\beta_2-\beta_1}B(1-\beta_1,\beta_1+\beta_2)\|\d_t^{\beta_1}\phi_1\|_{\C([0,T])},\\
|I_{1,2}|&\leq C(\Delta
t)^{\beta_1}|\d_t^{\beta_1}\phi_1(0)|\int_0^1
y^{-\beta_1}(1-y)^{\beta_1+\beta_2-1}dy\\
& \leq C |\d_t^{\beta_1}\phi_1(0)|(\Delta
t)^{\beta_1}B(1-\beta_1,\beta_1+\beta_2),\\
|I_{1,3}|&\leq
C(\Delta t)^{\beta_1+\beta_0}\int_0^{1}y^{-\beta_1}(1-y)^{\beta_2+\beta_0+\beta_1-1}dy\|\phi_1\|_{\C_{\beta_1}^{\beta_0}([0,T])}\\
& \leq C(\Delta t)^{\beta_1+\beta_0}
B(1-\beta_1,\beta_1+\beta_2+\beta_0)\|\phi_1\|_{\C_{\beta_1}^{\beta_0}([0,T])},
\end{align*}
which in turn entail
\begin{equation}\label{5.13}
|\Delta_t I_1|\leq
C\|\d_{t}^{\beta_2}\phi_2\|_{\C([0,T])}\|\phi_1\|_{\C_{\beta_1}^{\beta_0}([0,T])}(\Delta
t)^{\beta_3}.
\end{equation}
Coming to the estimating $I_2$, we have
\[
\Delta_t I_2=\sum_{j=1}^{4}I_{2,j},
\]
where we put
\begin{align*}
I_{2,1}&=\int_{t_1}^{t_2}\tau^{-\beta_1-1}I_{t}^{\beta_1}\d_t^{\beta_1}\phi_1(t_2-\tau)\{I_t^{\beta_2}[\d_t^{\beta_2}\phi_2-\d_t^{\beta_2}\phi_2(0)](t_2)
-I_t^{\beta_2}[\d_t^{\beta_2}\phi_2-\d_t^{\beta_2}\phi_2(0)](t_2-\tau)\}d\tau\\
I_{2,2}&=\int_{0}^{t_1}\tau^{-\beta_1-1}\{I_{t}^{\beta_1}[\d_t^{\beta_1}\phi_1-\d_t^{\beta_1}(0)](t_2-\tau)
-I_{t}^{\beta_1}[\d_t^{\beta_1}\phi_1-\d_t^{\beta_1}(0)](t_1-\tau)\}
\\&\times\{I_t^{\beta_2}[\d_t^{\beta_2}\phi_2-\d_t^{\beta_2}\phi_2(0)](t_1)
-I_t^{\beta_2}[\d_t^{\beta_2}\phi_2-\d_t^{\beta_2}\phi_2(0)](t_1-\tau)\}d\tau\\
I_{2,3}&=\int_0^{t_1}
\{I_t^{\beta_2}[\d_t^{\beta_2}\phi_2-\d_t^{\beta_2}\phi_2(0)](t_1)-I_t^{\beta_2}[\d_t^{\beta_2}\phi_2-\d_t^{\beta_2}\phi_2(0)](t_1-\tau)\}\\
&
\times \frac{\d_t^{\beta_1}\phi_1(0)[(t_2-\tau)^{\beta_1}-(t_1-\tau)^{\beta_1}]}{\tau^{\beta_1+1}\Gamma(1+\beta_1)}d\tau\\
I_{2,4}&=\int_{0}^{t_1}\frac{I_{t}^{\beta_1}\d_t^{\beta_1}\phi_1(t_2-\tau)}{\tau^{\beta_1+1}}
\{ I_t^{\beta_2}[\d_t^{\beta_2}\phi_2-\d_t^{\beta_2}\phi_2(0)](t_2)
-I_t^{\beta_2}[\d_t^{\beta_2}\phi_2-\d_t^{\beta_2}\phi_2(0)](t_2-\tau)\\&
-
I_t^{\beta_2}[\d_t^{\beta_2}\phi_2-\d_t^{\beta_2}\phi_2(0)](t_1)+I_t^{\beta_2}[\d_t^{\beta_2}\phi_2-\d_t^{\beta_2}\phi_2(0)](t_1-\tau)
\}d\tau.
\end{align*}
Then, appealing to the regularity of $\phi_i$ and exploiting
\eqref{2.1}, we deduce that
\begin{align}\label{5.14}\notag
|I_{2,1}|&\leq
C\|\phi_2\|_{\C_{\beta_2}^{\beta}([0,T])}\|\d_t^{\beta_1}\phi_1\|_{\C([0,T])}\int_{t_1}^{t_2}(t_2-\tau)^{\beta_1}\tau^{\beta_2+\beta-\beta_1-1}d\tau\\
\notag & \leq
C\|\phi_2\|_{\C_{\beta_2}^{\beta}([0,T])}\|\d_t^{\beta_1}\phi_1\|_{\C([0,T])}(\Delta
t)^{\beta_2+\beta-\beta_1},\\
\notag |I_{2,2}|&\leq C
|\phi_2\|_{\C_{\beta_2}^{\beta}([0,T])}\|\phi_1\|_{\C_{\beta_1}^{\beta_0}([0,T])}(t_2-t_1)^{\beta_0+\beta_1}
\int_0^{t_1}\tau^{\beta_2+\beta-\beta_1-1}d\tau\\
& \leq C(\Delta
t)^{\beta_0+\beta_1}|\phi_2\|_{\C_{\beta_2}^{\beta}([0,T])}\|\phi_1\|_{\C_{\beta_1}^{\beta_0}([0,T])},\\
\notag |I_{2,3}|&\leq
C|\d_t^{\beta_1}\phi_1(0)|\|\phi_2\|_{\C_{\beta_2}^{\beta}([0,T])}\int_{0}^{t_1}[(t_2-\tau)^{\beta_1}-(t_1-\tau)^{\beta_1}]\tau^{\beta_2+\beta-\beta_1-1}d\tau\\
\notag & \leq C(\Delta
t)^{\beta_1}|\d_t^{\beta_1}\phi_1(0)|\|\phi_2\|_{\C_{\beta_2}^{\beta}([0,T])}.
\end{align}
To evaluate $I_{2,4}$, we consider two different situations:
$2\Delta t\geq t_1$ and $2\Delta t< t_1$. In the first case, we have
\begin{align}\label{5.15}\notag
|I_{2,4}|&\leq
C\|\d_t^{\beta_1}\phi_1\|_{\C([0,T])}\|\phi_2\|_{\C_{\beta_2}^{\beta}([0,T])}\int_{0}^{2\Delta
t}\tau^{\beta+\beta_2-\beta_1-1}(t_2-\tau)^{\beta_1}d\tau\\
& \leq C(\Delta t)^{\beta+\beta_2-\beta_1}
\|\d_t^{\beta_1}\phi_1\|_{\C([0,T])}\|\phi_2\|_{\C_{\beta_2}^{\beta}([0,T])}.
\end{align}
Here, we again used \eqref{2.1} and the regularity of $\phi_i$.

If $\Delta t<t_1/2$, we rewrite $I_{2,4}$ as
\[
I_{2,4}=\sum_{j=1}^{4}\mathcal{I}_{j},
\]
where
\begin{align*}
\mathcal{I}_1&=\int_0^{2\Delta
t}\frac{I_t^{\beta_1}\d_t^{\beta_1}\phi_1(t_2-\tau)}{\tau^{\beta_1+1}}\{
I_t^{\beta_2}[\d_t^{\beta_2}\phi_2-\d_t^{\beta_2}\phi_2(0)](t_2)
-I_t^{\beta_2}[\d_t^{\beta_2}\phi_2-\d_t^{\beta_2}\phi_2(0)](t_2-\tau)\}d\tau,\\
\mathcal{I}_2&=\int_0^{2\Delta
t}\frac{I_t^{\beta_1}\d_t^{\beta_1}\phi_1(t_2-\tau)}{\tau^{\beta_1+1}}\{
I_t^{\beta_2}[\d_t^{\beta_2}\phi_2-\d_t^{\beta_2}\phi_2(0)](t_1-\tau)
-I_t^{\beta_2}[\d_t^{\beta_2}\phi_2-\d_t^{\beta_2}\phi_2(0)](t_1)\}d\tau,\\
\mathcal{I}_3&=\int^{t_1}_{2\Delta
t}\frac{I_t^{\beta_1}\d_t^{\beta_1}\phi_1(t_2-\tau)}{\tau^{\beta_1+1}}\{
I_t^{\beta_2}[\d_t^{\beta_2}\phi_2-\d_t^{\beta_2}\phi_2(0)](t_2)
-I_t^{\beta_2}[\d_t^{\beta_2}\phi_2-\d_t^{\beta_2}\phi_2(0)](t_1)\}d\tau,\\
\mathcal{I}_4&=\int^{t_1}_{2\Delta
t}\frac{I_t^{\beta_1}\d_t^{\beta_1}\phi_1(t_2-\tau)}{\tau^{\beta_1+1}}\{
I_t^{\beta_2}[\d_t^{\beta_2}\phi_2-\d_t^{\beta_2}\phi_2(0)](t_1-\tau)
-I_t^{\beta_2}[\d_t^{\beta_2}\phi_2-\d_t^{\beta_2}\phi_2(0)](t_2-\tau)\}d\tau.
\end{align*}
Arguing similarly to the previous case, we deduce that
\[
|\mathcal{I}_1|+|\mathcal{I}_2|\leq
C\|\phi_2\|_{\C_{\beta_2}^{\beta}([0,T])}\|\d_t^{\beta_1}\phi_1\|_{\C([0,T])}(\Delta
t)^{\beta+\beta_2-\beta_1}.
\]
As for $\mathcal{I}_3$, we have
\begin{align*}
|\mathcal{I}_3|&\leq
C(t_2-t_1)^{\beta+\beta_2}\|\phi_2\|_{\C_{\beta_2}^{\beta}([0,T])}\|\d_{t}^{\beta_1}\phi_1\|_{\C([0,T])}\int_{2\Delta}^{t_1}
\frac{(t_2-\tau)^{\beta_1}}{\tau^{\beta_1+1}}d\tau\\
&\leq C(\Delta
t)^{\beta}\|\phi_2\|_{\C_{\beta_2}^{\beta}([0,T])}\|\d_{t}^{\beta_1}\phi_1\|_{\C([0,T])}\int_{2\Delta
t}^{t_1}\tau^{\beta_2-\beta_1-1}d\tau\\
& \leq C(\Delta
t)^{\beta}\|\phi_2\|_{\C_{\beta_2}^{\beta}([0,T])}\|\d_{t}^{\beta_1}\phi_1\|_{\C([0,T])}.
\end{align*}
Clearly, the bound for $\mathcal{I}_4$ is analogous to the one of
$\mathcal{I}_3$. Thus, collecting all estimates for
$\mathcal{I}_{2,4}$, we arrive at
\[
|I_{2,4}|\leq C(\Delta
t)^{\beta}\|\phi_2\|_{\C_{\beta_2}^{\beta}([0,T])}\|\d_{t}^{\beta_1}\phi_1\|_{\C([0,T])},
\]
if $2\Delta t\geq t_1$.

Finally, taking into account this inequality along with \eqref{5.14}
and \eqref{5.15}, we end up with the bound
\[
|\Delta_t I_2|\leq C
\|\phi_2\|_{\C_{\beta_2}^{\beta}([0,T])}\|\phi_1\|_{\C^{\beta_0}_{\beta_1}([0,T])}[(\Delta
t)^{\beta_2+\beta-\beta_1}+(\Delta t)^{\beta_0+\beta_1}+(\Delta
t)^{\beta}++(\Delta t)^{\beta_1}].
\]
Combining this inequality with \eqref{5.13} completes the estimate
of $\langle J_{\beta_1}\rangle_{t,[0,T]}^{(\beta_3)}$, which in turn
finishes the proof of the point (iii)  and, accordingly, of this
claim.
\end{proof}
\begin{remark}\label{r5.1}
If $\phi_2\in\C^{1}([0,T])$, then making use of \cite[Theorem
3.1]{KPSV} and performing the straightforward calculations, we
easily conclude  that $\phi_2\in\C_{\beta_2}^{\beta}([0,T])$ for any
$\beta_2\in(0,1)$ and each $\beta\in(1-\beta_2,1)$. Thus, $\phi_2$
satisfies assumption stated in the point (iii) of Lemma \ref{l5.3}
and, accordingly, $\phi_1\phi_2\in\C_{\beta}^{\beta_3}([0,T])$ with
$\beta_3=\min\{\beta_0,\beta_1,1-\beta_1\}$ and
\[
\|\phi_1\phi_2\|_{\C_{\beta_1}^{\beta_3}([0,T])}\leq
C\|\phi_2\|_{\C^{1}([0.T])}\|\phi_1\|_{\C_{\beta_1}^{\beta_0}(0,T])}.
\]
\end{remark}
We conclude this section with description of the behavior of
$\d_t^{\nu_0}v$ and  $\d_t v$ at $t=0.$ It is worth noting that this
behavior will be a key point for the arguments providing the
nonvanishing $\d_t^{\nu_0}v(0)$ in the proof of Theorems
\ref{t3.1}-\ref{t3.2}.
\begin{lemma}\label{l5.4}
Let $v$ satisfy condition h1 and $r_i$ meet requirements h2 in the
case of the I type FDO and h3 and \eqref{3.0} in the case of the II
typed FDO. Then
\[
\d_t v|_{t=0}\neq \begin{cases}
0\qquad\qquad\qquad\qquad\qquad\text{for the I
type FDO,}\\
 v(0)\sum\limits_{i=0}^{M}\d_t^{\nu_i}r_i(0)\qquad\qquad \text{for
the II type FDO},
\end{cases}
\]
if and only if
\[
\d_t^{\nu_0}v|_{t=0}\neq 0.
\]
Moreover, the following equality holds
\begin{equation}\label{5.16}
r_0(0)\d_t^{\nu_0}v|_{t=0}=\begin{cases} \d_t
v|_{t=0}\qquad\qquad\qquad\qquad
\qquad\text{for the I type FDO,}\\
\d_tv|_{t=0}- v(0)\sum\limits_{i=0}^{M}\d_t^{\nu_i}r_i(0)\qquad
\text{for the II type FDO}.
\end{cases}
\end{equation}
\end{lemma}
\begin{proof}
It is apparent that the first statement in this lemma is a simple
consequence of equality \eqref{5.16}. Here, we focus on the
verification of \eqref{5.16} in the case of the II type FDO, since
the other case is completely analogous. Appealing to the properties
of $v$ (see h1) and making use of Lemma \ref{l5.2}, we readily get
the equalities
\[
\d_t^{\nu_i}v|_{t=0}=0,\qquad i=1,2,...,M.
\]
These equalities along with representation \eqref{5.8} with
$\phi_1=v$, $\phi_2=r_i$, $\beta_1=\nu_i$ arrive at the equalities
\[
\d_t^{\nu_i}(r_iv)|_{t=0}=v(0)\d_t^{\nu_i}r_i(0)+\frac{\nu_i}{\Gamma(1-\nu_i)}J_{\nu_i}(t;r_i,v)\Big|_{t=0},\quad
i=1,...,M,
\]
and
\[
\d_t^{\nu_0}(r_0v)|_{t=0}=r_0(0)\d_t^{\nu_0}v(0)+v(0)\d_t^{\nu_0}r_0(0)+\frac{\nu_0}{\Gamma(1-\nu_0)}J_{\nu_0}(t;r_0,v)\Big|_{t=0}
\]
which entail
\begin{equation}\label{5.17}
\d_tv|_{t=0}=r_0(0)\d_t^{\nu_0}v(0)+v(0)\sum_{i=0}^{M}\d_t^{\nu_i}r_i(0)+\sum_{i=0}^{M}\frac{\nu_i}{\Gamma(1-\nu_i)}J_{\nu_i}(t;r_i,v)\Big|_{t=0}.
\end{equation}
To handle the last sum in \eqref{5.17}, we use inequality
\eqref{5.12} with $\beta_1=\nu_i$, $\phi_1=v,$ $\phi_2=r_i$, which
end up with the equality
\[
\sum_{i=0}^{M}\frac{\nu_i}{\Gamma(1-\nu_i)}J_{\nu_i}(t;r_i,v)|_{t=0}=0.
\]
Collecting this equality with \eqref{5.17}, we achieve the desired
equality \eqref{5.16} and, accordingly, finish the proof of this
lemma.
\end{proof}
\begin{remark}\label{r5.2}
Obviously, if, in the case of the II type FDO, all
$r_i\in\C^{1}([0,T]),$ $i=0,...,M$, then equality \eqref{5.16} reads
as
\[
r_0(0)\d_t^{\nu_0}v|_{t=0}=\d_tv|_{t=0}.
\]
\end{remark}


\subsection{Completion of the proof of Theorems
\ref{t3.1}-\ref{t3.2}}\label{s5.2}

As for Theorem \ref{t3.1}, appealing to requirements on $r_i$ and
$v$ stated in h1-h2 and exploiting Lemma \ref{l5.4}, we derive that
\[
\d_t^{\nu_0}v|_{t=0}=\frac{\d_{t}v|_{t=0}}{r_0(0)}\neq 0.
\]
Taking into account these  inequalities and the regularity
$v\in\C_{\nu_0}^{\nu^*}([0,T^*])$, we use Lemma \ref{l5.1} with
$\phi=v$ and $\beta=\nu_0$ and, thus, we end up  with the desired
formula for $\nu_0$ and equality \eqref{3.1*} in the case of the I
type FDO.

Coming to Theorem \ref{t3.2}, we utilize the very similar arguments.
Namely, assumptions h1, h3 and  \eqref{3.0}  allow us to apply
Lemmas \ref{l5.3} and \ref{l5.4}, which provide that
$r_0v\in\C_{\nu_0}^{\nu^*}([0,T^*])$ and
\[
\d_t^{\nu_0}(r_0v)|_{t=0}=r_0(0)\d_t^{\nu_0}v|_{t=0}+v(0)\d_t^{\nu_0}r_0(0)=\d_tv|_{t=0}-v(0)\sum_{i=1}^{M}\d_t^{\nu_i}r_i(0)\neq
0.
\]
These relations mean that all assumptions of Lemma \ref{l5.1} with
$\phi=r_0v$ and $\beta=\nu_0$ hold, and, hence, this completes the
proof of Theorem \ref{t3.2}. \qed

\section{Proof of Theorems \ref{t4.1}-\ref{t4.5}}\label{s6}

\subsection{Proof of Theorem \ref{t4.1}}\label{s6.1} First, we
focus on the verification of this claim in the case of the local
measurement \eqref{i.5}. Since $u$ is a local classical solution of
\eqref{i.4}, \eqref{4.1} for all $t\leq t^*$, we exploit
\eqref{i.5},\eqref{4.7} and, appealing  to assumptions h8-h10,
deduce that
\begin{equation*}\label{6.0}
\d_t\psi|_{(x_0,0)}=g_0(x_0,0)+g(u_0)|_{(x_0,0)}+\mathcal{L}_1u_0|_{(x_0,0)}=\c_0.
\end{equation*}
 Next,
collecting nonvanishing of $\c_0$ stated in the assumption of
Theorem \ref{t4.1} with the regularity of $r_i$ (see h7), we
conclude that
\[
\d_t\psi|_{(x_0,0)}\neq 0,
\]
and besides, $r_i(x_0,t)$ satisfies h2 for the I type FDO and h3(i)
for the II type FDO. Thus, in force of Theorem \ref{t3.1} for the I
type FDO and Theorem \ref{t3.2} for the II type FDO, we end up with
\eqref{4.8} and the asymptotic behavior \eqref{3.1*} (for the I type
FDO) and \eqref{3.2*} (for the II type FDO) for $\psi(t)$ if
$t\in[0,t^*]$. Obviously, this asymptotic representation dictates
the uniqueness of the reconstruction of $\nu_0$ via the local
measurement \eqref{i.5}. Namely, considering the case of the I type
FDO and arguing by contradiction, we assume the existence of two
orders $\nu_0$ and $\bar{\nu}_0$ corresponding to the same
measurement $\psi(t)$ satisfying h9. Then, Theorem \ref{t3.1}
suggests the following relations for each $t\in[0,t^*]$:
\begin{align*}
&\frac{1}{\Gamma(\nu_0)}\int_{0}^{t}(t-\tau)^{\nu_0-1}[\d_t^{\nu_0}\psi(\tau)-\d_t^{\nu_0}\psi(0)]d\tau+t^{\nu_0}\c_0\\
& = \psi(t)-\psi(0)\\
& =
\frac{1}{\Gamma(\bar{\nu}_0)}\int_{0}^{t}(t-\tau)^{\bar{\nu}_0-1}[\d_t^{\bar{\nu}_0}\psi(\tau)-\d_t^{\bar{\nu}_0}\psi(0)]d\tau+t^{\bar{\nu}_0}\c_0.
\end{align*}
Bearing in mind the nonvanishing of $\c_0$ and letting, for
simplicity, $\nu_0<\bar{\nu}_0$, we have
\begin{equation}\label{6.1}
\c_0=\c_0 t^{\bar{\nu}_0-\nu_0}+S(t)\qquad \text{for all}\quad
t\in[0,t^*],
\end{equation}
where we set
\begin{align*}
S(t)&=
\frac{t^{-\nu_0}}{\Gamma(\bar{\nu}_0)}\int_{0}^{t}(t-\tau)^{\bar{\nu}_0-1}[\d_t^{\bar{\nu}_0}\psi(\tau)-\d_t^{\bar{\nu}_0}\psi(0)]d\tau
\\
& -
\frac{t^{-\nu_0}}{\Gamma(\nu_0)}\int_{0}^{t}(t-\tau)^{\nu_0-1}[\d_t^{\nu_0}\psi(\tau)-\d_t^{\nu_0}\psi(0)]d\tau.
\end{align*}
Thanks to the H\"{o}lder continuity of the corresponding  fractional
derivatives of $\psi$ (see h9), we derive that
\[
\underset{t\to 0}{\lim}\,S(t)=0.
\]
Employing this equality and passing to the limit in \eqref{6.1} as
$t\to 0$, we end up with the identity
\[
\c_0=0,
\]
which contradicts to the assumption on $\c_0$. This contradiction
can be resolved if we allow that $\nu_0=\bar{\nu}_0$.

As for the unique recovery of $\nu_0$ in the case of the II type
FDO, recasting the arguments above and exploiting asymptotic
\eqref{3.2*} for $\psi(t),$ we arrive at the desired results, which
complete the proof of Theorem \ref{t4.1} in the case of the local
observation \eqref{i.5}.

In fact, the verification of this theorem in the case of the
nonlocal measurement \eqref{i.6} repeats the arguments of the
previous step with using Theorem \ref{t3.2} instead of Theorem
\ref{t3.1}. Indeed, integrating equation \eqref{i.4} over $\Omega$
and appealing to \eqref{4.7}, we obtain
\begin{equation*}\label{6.2}
\d_t\psi|_{t=0}=\int_{\Omega}\mathcal{L}_1u_0|_{t=0}dx-\int_{\Omega}g(u_0)dx+\int_{\Omega}g_0(x,0)dx=\c_0.
\end{equation*}
Since $\c_0\neq 0$, the last equality along with h6-h7 allow us to
apply Theorem \ref{t3.2} and obtain the desired outcomes. That
completes the proof of Theorem \ref{t4.1}.\qed

\subsection{Verification of Theorems
\ref{t4.2}-\ref{t4.5}}\label{s6.2} Actually, the verification of
these theorems is a simple consequence of Theorem \ref{t4.1} and the
corresponding unique classical solvability of the direct problem
\eqref{i.4}, \eqref{4.1}-\eqref{4.4} established in
\cite{PSV,SV,V1,V2}. Indeed, for any given $\nu_0\in(0,1)$ and the
given data in the model satisfying  h5,h6,h8 and h11-h12 (for the
linear equation) and h11, h13-h17 with $\mathcal{K}\in\C^{1}([0,T])$
(for the semilinear equation), there exists a unique global
classical solution of these direct problems possessing the
regularity stated in Theorems \ref{t4.2}-\ref{t4.5}, respectively.
Thus, this fact provides the local classical solvability of
\eqref{i.4} for $t\in[0,t^*]$ which is, namely,  required in Theorem
\ref{t4.1}. At last, assumptions on the given data in Theorems
\ref{t4.2}-\ref{t4.5} guarantee the fulfillment of all conditions
required in Theorem \ref{t4.1}. The latter tells us that we can
apply Theorem \ref{t4.1} and end up with the unique $\nu_0$ computed
via \eqref{4.8}. At last, for the searched $\nu_0$, the previous
arguments provide the unique global classical solution $u$ having
the corresponding regularity. This leads to the desired results in
the corresponding theorems. \qed
\begin{remark}\label{r6.1}
Here, we give example of the requirements on the given functions
providing local classical solvability in the Cauchy problem
\eqref{3.3} (see Lemma \ref{l3.1}). Let's consider the inverse
problem \eqref{i.6}, \eqref{4.1}, \eqref{4.3} for linear equation
\eqref{i.4} with $\mathcal{L}_i$ and $\mathcal{N}_i$ given by
\eqref{4.10} with $a_0=a_0(t)$ and $b_0=b_0(t),$ $a_i=b_i=c_0=0,$
$i=1,...,n,$ and $r_i=r_i(t)$. Integrating \eqref{i.4} over $\Omega$
and, keeping in mind observation \eqref{i.6}, we conclude that
$\psi(t)$ solves in the classical sense the Cauchy problem for the
fractional differential equation similar to \eqref{3.3}, if  $t\leq
t^*$:
\[
\begin{cases}
\d_t\psi+\mathcal{K}*b_0\psi+a_0\psi=\int_{\Omega}g_0dx+\mathcal{I}(t)+(\mathcal{K}*\mathcal{I})(t),\\
\psi(0)=\int_{\Omega}u_0(x)dx,
\end{cases}
\]
where we set
\[
\mathcal{I}(t)=\begin{cases}
\int_{\partial\Omega}\varphi_1(x,t)dx,\qquad\text{if}\quad n\geq 2,\\
\varphi_1(\mathrm{l}_1,t)-\varphi_1(\mathrm{l}_2,t),\qquad\text{if}\quad
n=1,
\end{cases}
\]
with $\Omega=(\mathrm{l}_1,\mathrm{l}_2)$ in the one-dimensional
case. This fact suggests that conditions on $\mathcal{K},$ $b_0$ and
$g_0,$ $\varphi_1$ and $u_0$ stated in Theorems \ref{t4.3} guarantee
the unique  local classical solution to the Cauchy problem required
in Lemma \ref{l3.1}.
\end{remark}


\section{Numerical simulation}\label{s7}

\noindent Once the unique solvability of the inverse problems
\eqref{i.4}-\eqref{i.6}, \eqref{4.1}-\eqref{4.4} is established by
our main Theorems \ref{t4.1}-\ref{t4.5} and, in particular, the
explicit  formula \eqref{4.8} for $\nu_0$ is derived, one might like
to search $\nu_0$ explicitly. For this purpose, we focus on the most
often typical practical case when the regularity of the observation
$\psi$ in real life is less than it is required in Theorems
\ref{t4.2}-\ref{t4.5}. Namely, the continuous local or nonlocal
observation having H\"{o}lder continuous fractional derivatives is
rather exception than typical situation in practice. Indeed, in real
life, the measurements are often observed in noise-distorted
discrete forms. Thus, in order to exploit theoretically justified
formula \eqref{4.8} to compute $\nu_0$ in practice, we implement the
computational algorithm based on the Tikhonov regularization scheme
and the quasi-optimality approach, and then we demonstrate the work
of this algorithm via several numerical tests carried out in Section
\ref{s7.2}. We notice that in \cite{HPV,HPSV,KPSV2,PSV1}, the
similar numerical technique was incorporated to compute the order of
FDO via the so-called "logarithmic" formula rewritten here below in
our notations
\begin{equation}\label{7.1}
\nu_0=\begin{cases} \underset{t\to
0}{\lim}\frac{\ln\,|\psi(t)-\psi(0)|}{\ln\,
t}\qquad\qquad\quad\text{for the I
type FDO,}\\
\\
\underset{t\to 0}{\lim}\frac{\ln\,|r(t)\psi(t)-r(0)\psi(0)|}{\ln\,
t}\qquad\text{for the II type FDO,}
\end{cases}
\end{equation}
where we set
\begin{equation}\label{7.0}
r=r(t)=\begin{cases} r_0(x_0,t)\qquad\text{for the local
observation \eqref{i.5}}, \\
r_0(t)\qquad\quad\text{for the nonlocal observation \eqref{i.6}}.
\end{cases}
\end{equation}

\subsection{Algorithm of numerical reconstruction}\label{s7.1}
Assume that we observe the solution of \eqref{i.4},
\eqref{4.1}-\eqref{4.4} at the time moments $t_{k},\, k=1,2,...,K$,
ordered as $0<t_1<t_2<...<t_K\leq t^*,$ and denote
\begin{equation*}\label{7.2}
\psi_k=
\begin{cases}
u(x_0,t_k)\qquad\quad\text{for LO \eqref{i.5}},\\
\\
\int\limits_{\Omega}u(x,t_k)dx\qquad\text{for NLO \eqref{i.6}}.
\end{cases}
\end{equation*}
We also prescribe the presence of a noise observation
$\{\epsilon_{k}\}_{k=1}^{K}$ deteriorating these measurements, so
that, we observe
\begin{equation}\label{7.3}
\psi_{k,\epsilon}=\psi_k+\epsilon_k,\qquad\qquad k=1,2,...,K.
\end{equation}
At last, initial condition \eqref{4.1} suggests that
\[
\psi_{0}=\psi_{0,\epsilon}=
\begin{cases}
u_0(x_0)\qquad\quad\text{for LO \eqref{i.5}},\\
\\
\int\limits_{\Omega}u_0(x)dx\qquad\text{ for NLO \eqref{i.6}}.
\end{cases}
\]
In order to utilize the computational formula \eqref{4.8} (which
contains continuous-argument limit) in the case of discrete noisy
data \eqref{7.3}, we should reconstruct (approximately) the function
$\psi(t)$ from the values $\psi_{k,\varepsilon},$ $k=0,1,...,K.$
Appealing to the Tikhonov regularization scheme \cite{TG}, we
approximate $\psi(t)$ from the measurements $\{\psi_{k,\epsilon}\}$
by minimizing the  penalized least square functional
\begin{equation}\label{7.4}
\sum\limits_{k=0}^{K}[\psi_{k}-\psi_{k,\epsilon}]^2+\lambda\|\psi\|^{2}_{L^2_{t^{-\varrho}}(0,t_{K})}\to\min,
\end{equation}
where $\lambda$ is a regularization parameter. The selection of the
weighted space $L^{2}_{t^{-\varrho}}$ in this minimizer is dictated
by the asymptotic behavior of the measurement $\psi(t)$ stated in
Sections \ref{s4} and \ref{s6} (see also \eqref{3.1*} and
\eqref{3.2*})
\begin{equation}\label{7.4*}
\psi(t)=\psi_0+O(t^{\nu_0}),\qquad t\leq t^*.
\end{equation}
This equality tells us that the target function is (at least) square
integrable on $(0,t_{K})$ with an unbounded weight $t^{-\varrho},$
where the value $\varrho\in(0,1)$ is a user-defined.

Concerning an approximate minimizer to \eqref{7.4}, we search it in
the finite-dimensional form
\begin{equation}\label{7.5}
\psi_{\epsilon}(\lambda,t)=\sum_{j=1}^{\mathfrak{I}}a_{j}t^{b_{j}}
+\sum_{j=\mathfrak{I}+1}^{\mathfrak{P}}a_{j}\mathcal{P}_{j-\mathfrak{I}-1}^{(0,-\varrho)}(t/t_{K}),
\end{equation}
where
\[
\mathcal{P}_{j}^{(0,-\varrho)}(t/t_{K})=\sum_{i=0}^{j}\left(\begin{array}{c}
    j\\
    i
\end{array}\right)
\left(\begin{array}{c}
    j-\varrho\\
   j-i
\end{array}\right)\!(t/t_{K}-1)^{j-i}(t/t_{K})^{i},\quad
t\in(0,t_{K}),
\]
are an orthogonal system in $L_{t^{-\varrho}}^{2}(0,t_{K})$, and
power functions $t^{b_{i}}$, $i=1,2,\ldots,\mathfrak{I}$, are used
to facilitate capturing small-time asymptotics (see \eqref{7.4*}) of
the true observation, whereas
$b_{\mathfrak{I}}<b_{\mathfrak{I}-1}<\ldots<b_{1}$ play the role of
initial guesses for the $\nu_0$, if any. It is worth noting that the
selection of $b_i$ is user-defined, for example, it can be selected
from the uniform distribution on $(\frac{2\nu_1}{1-\alpha},1)$ in
the FDBC case (see h6) and on
$(\frac{2\max\{\nu_1,\bar{\nu}_1\}}{2-\alpha},1)$ in the remaining
cases (see h6, h16). As for the coefficients $a_j$ in \eqref{7.5},
they solve the system of linear algebraic equations:
\[
(\mathbb{Q}^{T}\mathbb{Q}+\lambda\mathbb{E})\mathbf{a}=\mathbb{Q}^{T}\mathbf{p}_{\epsilon},
\]
where we set
\begin{align*}
\mathbf{a}&=(a_{1},...,a_{\mathfrak{P}}),\quad
\mathbf{p}_{\epsilon}=(\psi_{0,\epsilon},\psi_{1,\epsilon},...,\psi_{K,\epsilon})^{T},\\
\mathbb{Q}&=\{q_{ij}\}^{K,\quad\mathfrak{P}}_{i=0,j=1},\quad
q_{ij}=h_{j}(t_{i}),\\
\mathbb{E}&=\{e_{l,m}\}_{l,m=1}^{\mathfrak{P}},\quad
e_{l,m}=\int_{0}^{t_{K}}t^{-\varrho}h_{l}(t)h_{m}(t)dt,\\
h_{l}(t)&=
\begin{cases}
t^{b_{l}},\qquad\qquad\qquad l=1,2,..,\mathfrak{I},\\
\mathcal{P}_{l-\mathfrak{I}-1}^{(0,-\varrho)}(t/t_{K}),\quad
l=1+\mathfrak{I},...,\mathfrak{P}.
\end{cases}
\end{align*}
This completes the approximate reconstruction of $\psi(t)$ in the
form of $\psi_{\epsilon}(\lambda,t)$.

Next, we have to compute the limit in formula \eqref{4.8}. Notice
that a numerical computation of such limits is (generally) an
ill-posed problem which requires a regularization technique too.
Clearly, the limit in \eqref{4.8} can be approximated by
\begin{equation}\label{7.6}
\nu_{0,\epsilon}(\lambda,\hat{t})=
\begin{cases}
\frac{\hat{t}[\psi_{\epsilon}(\lambda,\hat{t})-\psi_{0}]}{\int\limits_{0}^{\hat{t}}[\psi_{\epsilon}(\lambda,\tau)-\psi_0]d\tau}-1\qquad\qquad\text{
for
the I type FDO,}\\
\frac{\hat{t}[r(\hat{t})\psi_{\epsilon}(\lambda,\hat{t})-r(0)\psi_{0}]}{\int\limits_{0}^{\hat{t}}[r(\tau)\psi_{\epsilon}(\lambda,\tau)-r(0)\psi_0]d\tau}-1\qquad\text{for
the II type FDO}.
\end{cases}
\end{equation}
Here, a point $\hat{t}$ is chosen sufficiently close to zero and,
hence, this point may be considered also as a regularization
parameter.

In summary, we conclude that the regularized approximation
$\nu_{0,\epsilon}(\lambda,\hat{t})$ for $\nu_0$ demands the two
regularization parameters $\lambda$ and $\hat{t}$ which have to be
selected in a proper way. To achieve this,
 employing the quasi-optimality criterion
similarly to works \cite{HPV,KPSV2,PSV1}, we introduce two geometric
sequences of regularization parameters:
\[
\lambda=\lambda_i=\lambda_1\xi_1^{i-1},\qquad
i=1,2,...,K_1,\quad\text{and}\quad
\hat{t}=\hat{t}_{j}=\hat{t}_1\xi_2^{j-1},\quad j=1,2,...,K_2,
\]
with (user defined) values $\lambda_1$ and $\hat{t}_1$ and
$\xi_1,\xi_2\in(0,1)$. The quantities
$\nu_{0,\epsilon}(\lambda_i,\hat{t}_j)$ have to be computed for such
indices $i$ and $j$ via \eqref{7.6}. Then for each $\hat{t}_j$, we
should find $\lambda_{i_{j}}\in\{\lambda_i\}_{i=1}^{K_2}$:
\[
|\nu_{0,\epsilon}(\lambda_{i_{j}},\hat{t}_{j})-\nu_{0,\epsilon}(\lambda_{i_{j}-1},\hat{t}_{j})|=
\min\{|\nu_{0,\epsilon}(\lambda_{i},\hat{t}_{j})-\nu_{0,\epsilon}(\lambda_{i-1},\hat{t}_{j})|,\quad
i=2, 3,\ldots,K_{1}\}.\]
 Next, $\hat{t}_{j_{0}}$ is taken
from $\{\hat{t}_{j}\}_{j=1}^{K_{2}}$ such that
\[
|\nu_{0,\epsilon}(\lambda_{i_{j_0}},\hat{t}_{j_0})-\nu_{0,\epsilon}(\lambda_{i_{j_{0}-1}},\hat{t}_{j_{0}-1})|
=
\min\{|\nu_{0,\epsilon}(\lambda_{i_{j}},\hat{t}_{j})-\nu_{0,\epsilon}(\lambda_{i_{j-1}},\hat{t}_{j-1})|,\quad
j=2,3,\ldots,K_{2}\}.
\]
At last,  we put
$\nu_{\epsilon}^{I}:=\nu_{0,\epsilon}(\lambda_{i_{j_0}},\hat{t}_{j_{0}})$
being computed via \eqref{7.6} with $\lambda=\lambda_{i_{j_{0}}}$,
$\hat{t}=\hat{t}_{j_{0}}$ and, accordingly, this quantity is the
output of the proposed algorithm. In the next section, we not only
demonstrate the performance of this technique to reconstruct $\nu_0$
via formula \eqref{7.6} by noisy discrete measurements but also we
compare this numerical reconstruction of $\nu_0$ with the
regularized reconstruction based on the "logarithmic"
 formula \eqref{7.1}, that is
 \begin{equation}\label{7.1*}
\nu_{\epsilon}^{\ln}:=\nu_{0,\epsilon}^{\ln}(\lambda_{i_{j_0},\hat{t}_{j_0}})=
\begin{cases}
\frac{\ln|\psi_{\epsilon}(\lambda_{i_{j_0}},\hat{t}_{j_0})-\psi_0|}{\ln\,
\hat{t}_{j_0}}\qquad\qquad\qquad\text{for the I type FDO},\\
\\
\frac{\ln|r(\hat{t}_{j_0})\psi_{\epsilon}(\lambda_{i_{j_0}},\hat{t}_{j_0})-r(0)\psi_0|}{\ln\,
\hat{t}_{j_0}}\qquad\text{for the II type FDO},
\end{cases}
 \end{equation}
where $r(t)$ is defined in \eqref{7.0} and $\lambda_{i_{j_0}}$ and
$\hat{t}_{j_0}$ are chosen via the multiple quasi-optimality
criteria described above.


\subsection{Numerical tests}\label{s7.2}
In the forthcoming Example \ref{e7.1}, we consider one-dimensional
domain $\Omega$, whereas Example \ref{e7.2} is set in
two-dimensional domain. Moreover, Example \ref{e7.1} focuses on a
semilinear variant of \eqref{i.4} with both type FDO, while
Example \ref{7.2} discusses the linear version of \eqref{i.4} for
the I type FDO. In both examples, we consider the homogenous Neumann
boundary condition. Example \ref{e7.1} deals with the local
measurement at the point $x_0$, while Example \ref{e7.2} concerns
with the nonlocal observation. As for the noisy measurements, they
can be simulated by \eqref{7.6} with
\[
\epsilon_{k}=\mathfrak{G}(t_k),
\]
where the nonnegative function $\mathfrak{G}=\mathfrak{G}(t)$ has
the form
\begin{equation}\label{7.8}
\mathfrak{G}(t)=\varepsilon \begin{cases} t|\ln\,t|\qquad\quad
\qquad\text{the
first-type noise, \bf{N1}},\\
t^{\nu_0}\qquad\qquad\qquad\text{the second-type noise, \bf{N2}},\\
t^{\nu_0}|\ln\, t|\qquad\qquad \text{the third-type noise, \bf{N3}}
\end{cases}
\end{equation}
with some positive quantity $\varepsilon$. In our numerical
experiments, we select $\varepsilon=0.3$ and $0.03$ in Example
\ref{e7.1} and $\varepsilon=0.4$ and $0.04$ in Example \ref{e7.2}.
It is worth noting that the reasons to select $\mathfrak{G}(t)$ in
the form \eqref{7.8} are related to asymptotic \eqref{7.4*}  of
$\psi(t)$ and the fact that all calculations have to be performed at
the time point  very close to zero.   We refer readers to
\cite[Section 5]{PSV} and \cite[Section 6]{HPV}, for more details.

As for time moments $t_k$, we take 21 points, i.e. $t_{K}=t_{21},$
and test both uniform and nonuniform time distributions. Namely, in
Example \ref{e7.1}, we examine the following  nonuniform time
distribution:
\[
t_1=5\tau,\quad t_2=6\tau,\quad t_k=(9+k)\tau,\quad k=3,4,...,21,
\quad \tau=10^{-4},
\]
while in Example \ref{e7.2}, we analyze the uniform time
distribution in the form
\[
t_k=k\tau,\qquad k=1,2,...,21,\quad\tau=10^{-4}.
\]

Concenring the sequences of regularization parameters, $\lambda_i$
and $\hat{t}_j$, we chose
\[
\lambda_i=2^{1-i},\quad i=1,2,...,60,\quad\text{and}\quad
\hat{t}_j=2^{1-j}t_{21},\quad j=1,2,...,15.
\]
As for the approximate minimizer, it is selected in form \eqref{7.5}
with $\mathfrak{I}=3$ and $\mathfrak{P}=9$ (with $\varrho=0.99$);
the initial guess $b_1,$ $b_2$ and $b_3$ will be specified  in the
examples below.

The numerical outputs for the  examined examples are listed in
Tables \ref{tab7.1}-\ref{tab7.3}, where we use notations
$\nu^{\ln}_{\epsilon}=\nu^{\ln}_{0,\epsilon}(\lambda_{i,j_{0}},\hat{t}_{j_0})$
and
$\nu_{\epsilon}^{I}=\nu_{0,\epsilon}(\lambda_{i_{j_0}},\hat{t}_{j_0})$.
\begin{example}\label{e7.1}
Consider nonlinear equation \eqref{i.4} in the domain $\Omega=(0,1)$
and $t^*=t_{21},$ with the coefficients, orders and given functions
defined as
\begin{align*}
&a_{1 1}(x,t) = \cos\frac{\pi x}{ 4} +t,\;
  a_1(x,t) = -(x+t),\;\; a_0=0,\;\;   b_{1 1}(x,t) = t^{\frac{1}{3}}+\sin\pi x,\\
  & b_1=b_0=0,\;\; r_0(t)=1+t, \;\; r_1(t)=\frac{1}{2}, \;\; \gamma_1(t)=\frac{1+t^2}{2},\\
&
\nu_1=\frac{\nu_0}{3},\qquad\bar{\nu}_1=\frac{\nu_0}{2}\qquad(\text{i.e.}\quad
M_1=M_2=1)\\
   &u_{0}(x) = \cos(\pi x),\;\;\mathcal{K}(t)  =
   t^{\frac{-1}{3}},\quad
  g(u) = - x t \sin(u^2),\\
\intertext{while as the right-hand side $g_0(x,t)$ is taken}
g_0(x,t)&= \pi^2 \Big(\cos\frac{\pi x}{4} + t+ \frac{3t^{2/3}\sin(\pi x) }{2} + \frac{t \pi }{3\sin(\pi/3)}\Big)\cos(\pi x) \\
&- (x+t)\pi\sin(\pi x)  + 1+t +
\frac{t^{\nu_0-\nu_1}}{2\Gamma(1+\nu_0-\nu_1)} - \frac{(1+t^2)
t^{\nu_0-\bar{\nu}_1}}{2\Gamma(1+\nu_0-\bar{\nu}_1)} \\
&- x t \sin\bigg(\cos(\pi x)
+\frac{t^{\nu_0}}{\Gamma(1+\nu_0)}\bigg)^{2} \intertext{for the I
type FDO, and}
 g_0(x,t) &= \pi^2 \Bigl(\cos\frac{\pi x}{4} + t+ \frac{3t^{2/3}\sin(\pi x) }{2} + \frac{t \pi }{3\sin(\pi/3)}\Bigr)\cos(\pi x)\\
 & - x t \sin((\cos(\pi x)
  +t^{\nu_0}/\Gamma(1+\nu_0))^2)- (x+t)\pi\sin(\pi x) + 1\\
  & + \frac{t^{1-\nu_0}\cos(\pi x)}{\Gamma(2-\nu_0)} + (1+\nu_0) t + \frac{t^{\nu_0-\nu_1}}{2\Gamma(1+\nu_0-\nu_1)} \\
& - \frac{1}{2}\biggl(
\frac{t^{\nu_0-\bar{\nu}_1}}{\Gamma(1+\nu_0-\bar{\nu}_1)} +\frac{2
t^{2-\bar{\nu}_1} \cos(\pi x)}{\Gamma(3-\bar{\nu}_1)} +
\frac{(2+\nu_0)(1+\nu_0)
t^{2+\nu_0-\bar{\nu}_1}}{\Gamma(3+\nu_0-\bar{\nu}_1)} \biggr)
\intertext{for the II type FDO.}
\end{align*}
It is worth noting that the nonlinear term satisfies assumption
h14.(i). Besides, h16 is true for any $\alpha\in(0,1)$. Performing
straightforward  technical calculations, we arrive at the explicit
analytical solution
\begin{equation*}
\label{anal_sol_0} u(x,t) = \cos(\pi x) +
\frac{t^{\nu_0}}{\Gamma(1+\nu_0)}.
\end{equation*}
Initial guess in this test is selected as
$b_3=1.15\cdot\bar{\nu}_1,$ $b_2=1.45\cdot\bar{\nu}_1,$
$b_1=1.75\cdot\bar{\nu}_1$, so that the condition
\[
\frac{2\max\{\nu_1,\bar{\nu}_1\}}{2-\alpha}<b_3<b_2<b_1<1
\]
holds for all sufficiently small $\alpha$, e.g.
$\alpha<\frac{6}{23}$. As for the spatial observation point, we
choose $x_0=1/2$. Tables \ref{tab7.1}-\ref{tab7.2} display the
results of the proposed numerical algorithm for the different types
of noise.
\end{example}
\begin{table}[ht]
\centering \caption{Example \ref{e7.1} with the I type FDO}
\label{tab7.1}
 \footnotesize
\begin{tabular}{c|cc|cc|cc|cc|cc|cc}
\toprule
& \multicolumn{2}{c|}{$\varepsilon=0.03, \mathbf{N_1}$} & \multicolumn{2}{c|}{$\varepsilon=0.3, \mathbf{N_1}$} &
 \multicolumn{2}{c|}{$\varepsilon=0.03, \mathbf{N_2}$} & \multicolumn{2}{c|}{$\varepsilon=0.3, \mathbf{N_2}$} &
  \multicolumn{2}{c|}{$\varepsilon=0.03, \mathbf{N_3}$} & \multicolumn{2}{c}{$\varepsilon=0.3, \mathbf{N_3}$} \\
\midrule $\nu_0$ & $\nu_\epsilon^I$ & $\nu_\epsilon^{\ln}$ &
$\nu_\epsilon^I$ & $\nu_\epsilon^{\ln}$ & $\nu_\epsilon^I$ &
$\nu_\epsilon^{\ln}$ &
$\nu_\epsilon^I$ & $\nu_\epsilon^{\ln}$ & $\nu_\epsilon^I$ & $\nu_\epsilon^{\ln}$ & $\nu_\epsilon^I$ & $\nu_\epsilon^{\ln}$ \\
\midrule
0.1 & 0.1002 & 0.0922 & 0.1024 & 0.0915 & 0.1000 & 0.0880 & 0.1000 & 0.0537 & 0.0764 & 0.0661 & 0.0084 & -0.0691 \\
0.2 & 0.2005 & 0.1867 & 0.2042 & 0.1854 & 0.2001 & 0.1827 & 0.2000 & 0.1495 & 0.1771 & 0.1615 & 0.1088 & 0.0290 \\
0.3 & 0.3007 & 0.2831 & 0.3069 & 0.2805 & 0.2999 & 0.2793 & 0.2999 & 0.2467 & 0.2773 & 0.2586 & 0.2092 & 0.1277 \\
0.4 & 0.4010 & 0.3811 & 0.4111 & 0.3763 & 0.3995 & 0.3776 & 0.3998 & 0.3453 & 0.3773 & 0.3571 & 0.3090 & 0.2271 \\
0.5 & 0.5017 & 0.4804 & 0.5180 & 0.4715 & 0.4998 & 0.4774 & 0.4997 & 0.4452 & 0.4774 & 0.4569 & 0.4086 & 0.3270 \\
0.6 & 0.6027 & 0.5807 & 0.6248 & 0.5640 & 0.5995 & 0.5786 & 0.5995 & 0.5462 & 0.5773 & 0.5580 & 0.5082 & 0.4275 \\
0.7 & 0.7033 & 0.6814 & 0.7291 & 0.6507 & 0.6998 & 0.6811 & 0.6996 & 0.6482 & 0.6769 & 0.6602 & 0.6069 & 0.5284 \\
0.8 & 0.8018 & 0.7816 & 0.8184 & 0.7272 & 0.8003 & 0.7848 & 0.8001 & 0.7512 & 0.7768 & 0.7634 & 0.7061 & 0.6298 \\
0.9 & 0.8960 & 0.8796 & 0.8783 & 0.7890 & 0.8997 & 0.8896 & 0.8998 & 0.8550 & 0.8758 & 0.8676 & 0.8045 & 0.7315 \\
\bottomrule
\end{tabular}
\end{table}
\begin{table}[ht]
\centering \caption{Example \ref{e7.1} with the II type FDO}
\label{tab7.2}
 \footnotesize
\begin{tabular}{c|cc|cc|cc|cc|cc|cc}
\toprule
& \multicolumn{2}{c|}{$\varepsilon=0.03, \mathbf{N_1}$} & \multicolumn{2}{c|}{$\varepsilon=0.3, \mathbf{N_1}$} & \multicolumn{2}{c|}
{$\varepsilon=0.03, \mathbf{N_2}$} & \multicolumn{2}{c|}{$\varepsilon=0.3, \mathbf{N_2}$} & \multicolumn{2}{c|}{$\varepsilon=0.03, \mathbf{N_3}$}
& \multicolumn{2}{c}{$\varepsilon=0.3, \mathbf{N_3}$} \\
\midrule
$\nu_0$ & $\nu_\epsilon^I$ & $\nu_\epsilon^{\ln}$ & $\nu_\epsilon^I$ & $\nu_\epsilon^{\ln}$ & $\nu_\epsilon^I$ & $\nu_\epsilon^{\ln}$ &
$\nu_\epsilon^I$ & $\nu_\epsilon^{\ln}$ & $\nu_\epsilon^I$ & $\nu_\epsilon^{\ln}$ & $\nu_\epsilon^I$ & $\nu_\epsilon^{\ln}$ \\
\midrule
0.1 & 0.1010 & 0.0920 & 0.1031 & 0.0913 & 0.1008 & 0.0878 & 0.1003 & 0.0535 & 0.0772 & 0.0659 & 0.0092 & -0.0693 \\
0.2 & 0.2013 & 0.1865 & 0.2050 & 0.1851 & 0.2009 & 0.1825 & 0.2002 & 0.1492 & 0.1779 & 0.1613 & 0.1096 & 0.0288 \\
0.3 & 0.3016 & 0.2828 & 0.3078 & 0.2803 & 0.3007 & 0.2790 & 0.3008 & 0.2465 & 0.2781 & 0.2583 & 0.2100 & 0.1275 \\
0.4 & 0.4019 & 0.3808 & 0.4120 & 0.3761 & 0.4004 & 0.3773 & 0.4007 & 0.3451 & 0.3781 & 0.3568 & 0.3098 & 0.2269 \\
0.5 & 0.5026 & 0.4802 & 0.5189 & 0.4712 & 0.5007 & 0.4772 & 0.5006 & 0.4449 & 0.4783 & 0.4567 & 0.4095 & 0.3268 \\
0.6 & 0.6036 & 0.5805 & 0.6257 & 0.5637 & 0.6005 & 0.5784 & 0.6004 & 0.5459 & 0.5782 & 0.5577 & 0.5091 & 0.4273 \\
0.7 & 0.7042 & 0.6812 & 0.7300 & 0.6505 & 0.7007 & 0.6809 & 0.7006 & 0.6480 & 0.6778 & 0.6599 & 0.6078 & 0.5282 \\
0.8 & 0.8028 & 0.7814 & 0.8193 & 0.7270 & 0.8013 & 0.7846 & 0.8010 & 0.7509 & 0.7777 & 0.7632 & 0.7071 & 0.6296 \\
0.9 & 0.8970 & 0.8794 & 0.8792 & 0.7888 & 0.9007 & 0.8894 & 0.9007 & 0.8548 & 0.8767 & 0.8673 & 0.8054 & 0.7313 \\
\bottomrule
\end{tabular}
\end{table}
\begin{example}\label{e7.2}
Here, we consider the linear variant of equation \eqref{i.4} in
two-dimensional domain $\Omega=(0,2)\times(0,2),$ $T=1,$
\[
\frac{1}{2}\d_t^{\nu_0}u-\frac{(1+t^{2})}{4}\d_t^{\nu_1}-\Delta
u-\mathcal{K}*(\Delta
u+4u)=\sum_{i=1}^{3}g_{i}(x,t)\quad\text{in}\quad\Omega_{T},
\]
where $\nu_1=\frac{\nu_0}{5}$;  the memory kernel  and the initial
data are given by
\[
\mathcal{K}(t)=t^{-\gamma}+t^{1-\gamma}\quad\text{with any}\quad
\gamma\in(0,1), \quad \text{and}\quad
u_0(x_1,x_2)=x_1^2x_2^2(2-x_1)^{2}(2-x_2)^{2},
\]
$g_i$ in the right-hand side in the equation are given by
\begin{align*}
&g_{1}(x,t)=x_{2}^{2}x_{1}^{2}(2-x_{1})^{2}(2-x_{2})^{2}\bar{g}_{1}(t),\\
&g_{2}(x,t)=-4[x_{2}^{2}(2-x_{2})^{2}(2-6x_{1}+3x_{1}^{2}) +
x_{1}^{2}(2-x_{1})^{2}(2-6x_{2}+3x_{2}^{2})
]\bar{g}_{2}(t),\\
& g_{3}(x,t)=-4[x_{2}^{2}(2-x_{2})^{2}(2-6x_{1}+3x_{1}^{2}) +
x_{1}^{2}(2-x_{1})^{2}(2-6x_{2}+3x_{2}^{2})
+x_{2}^{2}x_{1}^{2}(2-x_{1})^{2}(2-x_{2})^{2} ]\bar{g}_{3}(t)
\end{align*}
with
\begin{align*}
 &
\bar{g}_{1}(t)=\frac{\Gamma(1+\nu_0)}{2}-\frac{(1+t^{2})t^{4\nu_1}\Gamma(1+\nu_0)}{4\Gamma(1+4\nu_1)}
,\quad \bar{g}_{2}(t)=2+t^{\nu_0},\\
&
\bar{g}_{3}(t)=\frac{2\,t^{1-\gamma}}{1-\gamma}+\frac{2\,t^{2-\gamma}}{2-\gamma}
+
\frac{\Gamma(1-\gamma)\Gamma(1+\nu_0)}{\Gamma(2-\gamma+\nu_0)}t^{1+\nu_0-\gamma}
+
\frac{\Gamma(2-\gamma)\Gamma(1+\nu_0)}{\Gamma(3-\gamma+\nu_0)}t^{2-\gamma+\nu_0}.
\end{align*}
 The straightforward computations give the explicit
solution
\[
u(x_1,x_2,t)=x_{2}^{2}x_{1}^{2}(2-x_{1})^{2}(2-x_{2})^{2}[2+t^{\nu_0}]
\]
of this initial-boundary value problem. Recall that in this example
we test the nonlocal observation $\psi(t)=\int_{\Omega}u(x,t)dx$,
where
\[
\psi_0=\int_{\Omega}u_{0}(x_1,x_2)dx_1dx_2=\frac{512}{225}.\]
In
this example, we chose initial guess $b_i$ in the form
\[
b_1=1.75\cdot\nu_1,\quad b_2=1.45\cdot\nu_1,\quad
b_3=1.15\cdot\nu_1,
\]
which also allows us to satisfy the condition
$b_3>\frac{2\nu_1}{2-\alpha}$ for enough small $\alpha$. The
outcomes of this numerical test are listed in Table \ref{tab7.3}.
\end{example}
\begin{table}[ht]
\centering \caption{Example \ref{e7.2}} \label{tab7.3} \footnotesize
\begin{tabular}{c|cc|cc|cc|cc|cc|cc}
\toprule
& \multicolumn{2}{c|}{$\varepsilon=0.04, \mathbf{N_1}$} & \multicolumn{2}{c|}{$\varepsilon=0.4, \mathbf{N_1}$}
 & \multicolumn{2}{c|}{$\varepsilon=0.04, \mathbf{N_2}$} & \multicolumn{2}{c|}{$\varepsilon=0.4, \mathbf{N_2}$}
 & \multicolumn{2}{c|}{$\varepsilon=0.04, \mathbf{N_3}$} & \multicolumn{2}{c}{$\varepsilon=0.4, \mathbf{N_3}$} \\
\midrule $\nu_0$ & $\nu_\epsilon^I$ & $\nu_\epsilon^{\ln}$ &
$\nu_\epsilon^I$ & $\nu_\epsilon^{\ln}$ & $\nu_\epsilon^I$ &
$\nu_\epsilon^{\ln}$ &
$\nu_\epsilon^I$ & $\nu_\epsilon^{\ln}$ & $\nu_\epsilon^I$ & $\nu_\epsilon^{\ln}$ & $\nu_\epsilon^I$ & $\nu_\epsilon^{\ln}$ \\
\midrule
0.1 & 0.1024 & 0.0915 & 0.1032 & 0.0787 & 0.1000 & 0.0744 & 0.1000 & 0.0326 & 0.0720 & 0.0481 & 0.0014 & -0.1045 \\
0.2 & 0.2042 & 0.1854 & 0.2056 & 0.1777 & 0.1999 & 0.1744 & 0.1999 & 0.1326 & 0.1719 & 0.1481 & 0.1007 & -0.0045 \\
0.3 & 0.3069 & 0.2805 & 0.3096 & 0.2758 & 0.2998 & 0.2744 & 0.2998 & 0.2326 & 0.2717 & 0.2481 & 0.2000 & 0.0955 \\
0.4 & 0.4111 & 0.3763 & 0.4155 & 0.3724 & 0.3997 & 0.3744 & 0.3994 & 0.3326 & 0.3716 & 0.3481 & 0.2992 & 0.1955 \\
0.5 & 0.5180 & 0.4715 & 0.5251 & 0.4660 & 0.4998 & 0.4744 & 0.5004 & 0.4326 & 0.4717 & 0.4481 & 0.3991 & 0.2955 \\
0.6 & 0.6248 & 0.5640 & 0.6332 & 0.5546 & 0.6006 & 0.5744 & 0.5988 & 0.5326 & 0.5721 & 0.5481 & 0.4985 & 0.3955 \\
0.7 & 0.7291 & 0.6507 & 0.7365 & 0.6351 & 0.7011 & 0.6744 & 0.6999 & 0.6326 & 0.6724 & 0.6481 & 0.5984 & 0.4955 \\
0.8 & 0.8184 & 0.7272 & 0.8217 & 0.7037 & 0.8012 & 0.7744 & 0.8011 & 0.7326 & 0.7721 & 0.7481 & 0.6982 & 0.5955 \\
0.9 & 0.8783 & 0.7890 & 0.8751 & 0.7574 & 0.9006 & 0.8744 & 0.8992 & 0.8326 & 0.8722 & 0.8481 & 0.7976 & 0.6955 \\
\bottomrule
\end{tabular}
\end{table}

\textit{Analyzing the numerical outcomes presented in Tables
\ref{tab7.1}-\ref{tab7.3}, we conclude that the exploited numerical
algorithm is very effective (in practice) in the finding order
$\nu_0$. Moreover, all results demonstrate that formula \eqref{4.8}
(and correspondingly \eqref{7.6}) provides higher accuracy  than
formula \eqref{7.1} (and \eqref{7.1*}) proposed in
\cite{HPV,HPSV,PSV} for the multi-term fractional in time diffusion
equation \eqref{i.4} in the case of both the local and the nonlocal
observation. The latter suggests to use this formula to recover
$\nu_0$ for the simultaneous reconstruction of several scalar
parameters in $\d_t$ (see for details \cite{HPV,HPSV}).}



\end{document}